\documentclass[a4paper,11pt,reqno]{amsart}
\usepackage[left=26mm,right=26mm,top=30mm,bottom=30mm]{geometry}

\usepackage{amssymb}
\usepackage[shortlabels]{enumitem}
\usepackage{bm}
\usepackage{mathtools}
\usepackage{mathrsfs}
\usepackage{array} 
\usepackage{xcolor} 
\usepackage{xparse}
\usepackage[sort&compress,numbers]{natbib}
\usepackage{caption}
\newtheorem{theorem}{Theorem}[section]
\theoremstyle{plain}
\newtheorem{definition}{Definition}[section]
\newtheorem{lemma}{Lemma}[section]

\newtheorem{corollary}{Corollary}[section]
\newtheorem{remark}{Remark}[section]
\numberwithin{equation}{section}

\makeatletter
\newcommand\makebig[2]{%
 \@xp\newcommand\@xp*\csname#1\endcsname{\bBigg@{#2}}%
 \@xp\newcommand\@xp*\csname#1l\endcsname{\@xp\mathopen\csname#1\endcsname}%
 \@xp\newcommand\@xp*\csname#1r\endcsname{\@xp\mathclose\csname#1\endcsname}%
}
\makeatother

\makebig{biggg} {3.0}
\makebig{Biggg} {3.5}
\makebig{bigggg}{4.0}
\makebig{Bigggg}{4.5}
\makebig{biggggg}{5}
\makebig{Biggggg}{5.5}
\makebig{bigggggg}{6}
\makebig{Bigggggg}{6.5}

\makeatletter
\newcommand{\doublehat}[1]{%
\begingroup%
  \let\macc@kerna\z@%
  \let\macc@kernb\z@%
  \let\macc@nucleus\@empty%
  \hat{\mathchoice%
    {\raisebox{.2ex}{\vphantom{\ensuremath{\displaystyle #1}}}}%
    {\raisebox{.2ex}{\vphantom{\ensuremath{\textstyle #1}}}}%
    {\raisebox{.16ex}{\vphantom{\ensuremath{\scriptstyle #1}}}}%
    {\raisebox{.14ex}{\vphantom{\ensuremath{\scriptscriptstyle #1}}}}%
    \smash{\hat{#1}}}%
\endgroup%
}
\makeatother

\begin{document}
\title{The Riesz basisness of the eigenfunctions and eigenvectors connected to the stability problem of a fluid-conveying tube with boundary control}
\author{Mahyar Mahinzaeim$^{\tt1,\ast}$}
\author{Gen Qi Xu$^{\tt2}$}
\author{Xiao Xuan Feng$^{\tt1,\tt3}$}

 \thanks{
\vspace{-1em}\newline\noindent
\newline\noindent
$^{\tt1}$ Research Center for Complex Systems, Aalen University, Germany.
 \newline\noindent
 $^{\tt2}$ Department of Mathematics, Tianjin University, China.
 \newline\noindent
 $^{\tt3}$ School of Mathematics and Information Science, Hebei Normal University of Science and Technology, China.
 \newline\noindent
 {\sc Emails}:
 {\tt m.mahinzaeim@web.de},~{\tt gqxu@tju.edu.cn},~{\tt xiaoxuan.feng@hs-aalen.de}.
 \newline\noindent
$^{\ast}$ Corresponding author.
}

\begin{abstract}
In the present paper we study the stability problem for a stretched tube conveying fluid with boundary control. The abstract spectral problem concerns operator pencils of the forms
\begin{equation*}
\mathcal{M}\left(\lambda\right)=\lambda^2G+\lambda D+C\quad\text{and}\quad\mathcal{P}\left(\lambda\right)=\lambda I-T
\end{equation*}
taking values in different Hilbert product spaces. Thorough analysis is made of the existence, location, multiplicities, and asymptotics of eigenvalues in the complex plane and Riesz basisness of the corresponding eigenfunctions and eigenvectors. Well-posedness of the closed-loop system represented by the initial-value problem for the abstract equation
\begin{equation*}
\dot{{x}}\left(t\right)=Tx\left(t\right)
\end{equation*}
is established in the framework of $C_0$-semigroups as well as expansions of the solutions in terms of eigenvectors and stability of the closed-loop system. For the parameters of the problem we give new regions, larger than those in the literature, in which a stretched tube with flow, simply supported at one end, with a boundary controller applied at the other end, can be exponentially stabilised.
\end{abstract}
\maketitle

\pagestyle{myheadings} \thispagestyle{plain} \markboth{\sc M.\ Mahinzaeim, G.\ Q.\ Xu, X.\ X.\ Feng}{\sc Exponential stability of boundary-controlled tube conveying fluid}

\section{Introduction and statement of the problem}\label{sec_intro}

The partial differential equation governing the motion of a long and thin homogeneous tube of unit length, carrying the stationary flow of an incompressible fluid can be written in terms of the transverse deflection $\bm{w}\left(s,t\right)$ for $s\in\left[0,1\right]$, $t\in\mathbf{R}_+$ as
\begin{equation}\label{eq_01}
\frac{\partial^4\bm{w}\left(s,t\right)}{\partial s^4}+\eta^2\frac{\partial^2\bm{w}\left(s,t\right)}{\partial s^2}+2\beta \eta\frac{\partial^2\bm{w}\left(s,t\right)}{\partial s\partial t}+\frac{\partial^2\bm{w}\left(s,t\right)}{\partial t^2}=0.
\end{equation}
Here $\eta\geq 0$ represents the velocity of the fluid flow and $\beta\in\left(0,1\right)$ is a parameter depending only on the tube and fluid mass densities. This is the basic type of tube to model mathematically and for which a vast literature exists. It can be used to model flow-induced vibrations e.g.\ in the case of fully developed turbulent flow -- i.e.\ whenever the flow resembles plug flow. Classic examples include vibrations in nuclear reactor fuel pins, offshore risers and tendons, wind-induced vibrations in electric power lines, and supersonic panel flutter. There are other examples, see the recent survey \cite{PAIDOUSSIS2022103664}. A full derivation of \eqref{eq_01} and an explanation of the physics involved may be found e.g.\ in \cite[Chapter 3]{Paidoussis2014} or \cite[Section 8.5]{CrandallEtAl1968}. (The reader should be aware that apart from the main model approximation that the tube is ``long and thin'', so a spatially one-dimensional model is valid, the model \eqref{eq_01} relies only on the further assumption of mean fluid-flow velocity. Thus it should be a reasonable representation even for laminar flow. It is also noteworthy that the term ``vibration'' is used herein to include both flutter and buckling behaviour.)

Several variants of \eqref{eq_01} can be obtained (the following list is by no means complete). For example, when there is no flow in the tube, $\eta=0$, the equation is simplest and reduces to the model for the classic Euler--Bernoulli beam. On the other hand, if the tube is subjected to an external axial force which is negatively proportional to the bending moment, then the equation analogous to \eqref{eq_01} includes an extra term
\begin{equation*}\label{eq_01xcv}
-\gamma\frac{\partial^2\bm{w}\left(s,t\right)}{\partial s^2},
\end{equation*}
where $\gamma$ is positive or negative, depending on whether the force is tensile or compressive, respectively. If the tube is made of a material modelled by the Kelvin--Voigt model for linear viscoelasticity and is situated in a linearly viscous surrounding medium, then the equation has an extra term
\begin{equation*}
\alpha\frac{\partial^5 \bm{w}\left(s,t\right)}{\partial s^4\partial t}+\delta\frac{\partial \bm{w}\left(s,t\right) }{\partial t},
\end{equation*}
where $\alpha>  0$ is the viscoelastic damping coefficient and the parameter $\delta\geq 0$ corresponds to the viscous damping due to friction from the surrounding medium. Referring the reader to our paper \cite{MahinzaeimEtAl2021b}, we will not consider this situation further here and will, in fact, focus our attention on the tension modification, through addition of the term indicated above for $\gamma> 0$ to \eqref{eq_01}, i.e.\
\begin{equation}\label{eq_01xabvf}
\frac{\partial^4\bm{w}\left(s,t\right)}{\partial s^4}-(\gamma-\eta^2)\,\frac{\partial^2\bm{w}\left(s,t\right)}{\partial s^2}+2\beta \eta\frac{\partial^2\bm{w}\left(s,t\right)}{\partial s\partial t}+\frac{\partial^2\bm{w}\left(s,t\right)}{\partial t^2}=0,
\end{equation}
and suppose damping to be present only at the boundary of the vibrating fluid-conveying tube (and not from insertion of damping terms into its equation of motion); we will discuss this latter matter shortly.

For this tube system \eqref{eq_01xabvf} let us pose the following initial/boundary-value problem. We associate with \eqref{eq_01xabvf} given initial conditions
\begin{equation}\label{eq_04}
\bm{w}\left(s,0\right)=g\left(s\right),\quad \left.\frac{\partial \bm{w}\left(s,t\right)}{\partial t}\right|_{t=0}=h\left(s\right),
\end{equation}
where the functions $g$, $h$ are assumed suitably smooth in a sense to be made more precise later (in Section \ref{sec22}). For boundary conditions we assume that the end of the tube at $s=0$ is simply supported,
\begin{equation}\label{eq_02b}
\bm{w}\left(0,t\right)=\left.\frac{\partial^2 \bm{w}\left(s,t\right)}{\partial s^2}\right|_{s=0}=0,
\end{equation}
and at the end $s=1$ we assume that
\begin{equation}\label{eq_06}
\left.\frac{\partial^2 \bm{w}\left(s,t\right)}{\partial s^2}\right|_{s=1}=\bm{u}\left(t\right),\quad \left.\frac{\partial^3 \bm{w}\left(s,t\right)}{\partial s ^3}\right|_{s =1}= \left.(\gamma-\eta^2)\, \frac{\partial \bm{w}\left(s,t\right)}{\partial s }\right|_{s =1}.
\end{equation}
In the former condition in \eqref{eq_06} we take account of a boundary controller $\bm{u}\left(t\right)$ corresponding to the bending moment of the tube at $s=1$. The latter condition in \eqref{eq_06} assumes that the tube undergoes tension which acts always in a fixed direction along the axis of the tube and does not rotate with its end. This is a particularly interesting case because when there is no flow in the tube, the uncontrolled system is the same as in \textit{Beck's Problem} (the mixed derivative, or gyroscopic, term vanishes in \eqref{eq_01xabvf}, upon taking $\eta=0$, see \cite{Beck1952,Bolotin1963,Ziegler1977}) but under the action of a ``nonfollowing'' axial force, corresponding to an energy-conservative system. So we are led to suspect here that by demanding that the controls $\bm{u}$ be of the negative velocity feedback variety, i.e.\ boundary damping, it should be possible, in the case $\eta>0$ such that $\gamma> \eta^2$, to control and thereby to stabilise the vibration of the tube, meaning that, for any initial conditions, the total vibrational energy dissipates in time. 

Let us see what we can develop along these lines. We begin by assuming for the moment that there exists a (classical) solution $\bm{w}\left(s,t\right)$ of \eqref{eq_01xabvf}--\eqref{eq_06}. If at a given time $t$ the total energy of the tube system is given by
\begin{equation}\label{eqaas345}
\mathbf{E}\left(t\right)=\frac{1}{2}\int_0^1\left[\left(\frac{\partial^2 \bm{w}\left(s,t\right)}{\partial s^2}\right)^2+(\gamma-\eta^2)\left(\frac{\partial \bm{w}\left(s,t\right)}{\partial s}\right)^2+\left(\frac{\partial \bm{w}\left(s,t\right)}{\partial t}\right)^2\right]ds,
\end{equation}
then by differentiating \eqref{eqaas345} with respect to $t$, followed by some integration by parts and use of \eqref{eq_01xabvf} and \eqref{eq_02b}, \eqref{eq_06}, we find that the time rate of change of the energy is
\begin{equation*}
\frac{d}{dt}\mathbf{E}\left(t\right)=-2\beta\eta\int_0^1\frac{\partial^2 \bm{w}\left(s,t\right)}{\partial s\partial t}\,\frac{\partial \bm{w}\left(s,t\right)}{\partial t}\,ds+\left.\bm{u}\left(t\right)\frac{\partial^2 \bm{w}\left(s,t\right)}{\partial s\partial t}\right|_{s =1}.
\end{equation*}
Using that
\begin{equation*}
-2\beta\eta\int_0^1\frac{\partial^2 \bm{w}\left(s,t\right)}{\partial s\partial t}\,\frac{\partial \bm{w}\left(s,t\right)}{\partial t}\,ds=-\beta\eta\int_0^1\frac{\partial }{\partial s}\left(\frac{\partial \bm{w}\left(s,t\right)}{\partial t}\right)^2ds=-\beta\eta\left.\left(\frac{\partial \bm{w}\left(s,t\right)}{\partial t}\right)^2\right|_{s =1},
\end{equation*}
we obtain
\begin{equation}\label{eq12new34}
\frac{d}{dt}\mathbf{E}\left(t\right)=-\beta\eta\left.\left(\frac{\partial \bm{w}\left(s,t\right)}{\partial t}\right)^2\right|_{s =1}+\left.\bm{u}\left(t\right)\frac{\partial^2 \bm{w}\left(s,t\right)}{\partial s\partial t}\right|_{s =1}.
\end{equation}
This prompts the selection of $\bm{u}\left(t\right)$ as 
\begin{equation}\label{eq_06ssyy}
\bm{u}\left(t\right)=\left.-\kappa\frac{\partial^2 \bm{w}\left(s,t\right)}{\partial s \partial t}\right|_{s=1},
\end{equation}
a feedback controller which is determined from knowledge of the angular velocity of the tube at $s=1$, for some feedback parameter $\kappa\geq 0$. Insertion of \eqref{eq_06ssyy} into \eqref{eq12new34} yields (the inequality a consequence of  $\beta\in\left(0,1\right)$ and $\eta,\kappa\geq0$)
\begin{equation}\label{eq_06ssyxy}
\frac{d}{dt}\mathbf{E}\left(t\right)=-\beta\eta\left.\left(\frac{\partial \bm{w}\left(s,t\right)}{\partial t}\right)^2\right|_{s =1}-\kappa\left.\left(\frac{\partial^2 \bm{w}\left(s,t\right)}{\partial s\partial t}\right)^2\right|_{s =1}\leq 0.
\end{equation}
Consequently, if $\bm{w}\left(s,t\right)$ is a solution of the initial/boundary-value problem for \eqref{eq_01xabvf} with $\bm{u}\left(t\right)$ given by \eqref{eq_06ssyy}, the energy is nonincreasing in time and we have a dissipative system (obviously this also holds for $\bm{u}\equiv 0$ and is not really a separate case from ours); the energy is conserved in the case $\eta=\kappa=0$. The dissipativity inequality \eqref{eq_06ssyxy} alone, of course, gives us no idea of the asymptotic behaviour of the system in time, e.g., as to how fast the energy of solutions of \eqref{eq_01xabvf}--\eqref{eq_06} decays as time goes to infinity. In fact, it tells us nothing about whether energy decay takes place at all, i.e., whether $\mathbf{E}\left(t\right)\rightarrow 0$ as $t\rightarrow\infty$. In this paper we will establish that under certain conditions on the parameter triple $\left\{\gamma,\eta,\kappa\right\}$ and for arbitrary $\beta\in\left(0,1\right)$ this is in fact true and, moreover, that the energy decay actually is \textit{exponential}. This is the main business of the paper which is worked out by studying the boundary-eigenvalue problem corresponding to \eqref{eq_01xabvf}--\eqref{eq_06} to obtain parameter regions (of some independent interest in engineering circles!) larger than those already obtained in the literature in order to produce an exponential decay rate. To explain what is meant by this and provide theoretical justification for this ``spectral approach'', a certain amount of background is required. 

\section{Some background and preliminary results}\label{sec22}

To avoid any confusion, following standard terminology, we will use from now on the acronym $\mathbf{CLS}$ (closed-loop system) to denote the initial/boundary-value problem posed by \eqref{eq_01xabvf}--\eqref{eq_06} with $\bm{u}\left(t\right)$ as given by the boundary feedback relation \eqref{eq_06ssyy}. As mentioned at the end of the Introduction, we will be concerned with the exponential stabilisability or stability of (solutions of) the $\mathbf{CLS}$ (we define in Section \ref{sec2_3} what we mean, in more detail, by stability and exponential stability). For this it is instructive first to consider, in the context of the existing literature, a number of uncontrolled variants of the above system which can be exemplified by the initial/boundary-value problem for \eqref{eq_01} subject to cantilevered, fully clamped or simply supported boundary conditions; by ``cantilevered'' we mean that the tube is clamped at one end and left free at the other, corresponding to the boundary conditions $\bm{w}\left(0,t\right)=\left.\left(\partial \bm{w}/\partial s\right)\left(s,t\right)\right|_{s =0}=0$ and $\left.(\partial^2 \bm{w}/\partial s^2)\left(s,t\right)\right|_{s =1}=\left.(\partial^3 \bm{w}/\partial s^3)\left(s,t\right)\right|_{s =1}=0$. It is convenient to refer to this initial/boundary-value problem with any of the aforementioned boundary conditions as the open-loop system, abbreviated hereafter to \textbf{OLS}.

There is a long history in engineering and mathematics as regards the stability properties of the \textbf{OLS}, which, in general, is investigated either through use of so-called multiplier methods \cite{Komornik1994,MR3220858}, a generalisation of Liapunov's direct method to partial differential equations \cite{MR113375,MR113375x}, or through a spectral approach by the use of separation of variables in the \textbf{OLS} followed by a careful analysis of the eigenvalues of the resulting boundary-eigenvalue problem under variations in the pair $\left\{\beta,\eta\right\}$. The first investigations have been due, independently, to Feodosiev \cite{Feodosiev1951} and Housner \cite{Housner1952}. Subsequently several other investigations, along with various model extensions, have been presented in many papers, \cite{Heinrich1956,Benjamin1961b,PlautHuseyin1975,paidoussis_1966b,paidoussis_1966,PaidoussisIssid1974,Dotsenko1979,Movchan1965,Handelman1955,BajajEtAl1980,MR122199,Holmes1978,Holmes1977,MR1211620,GregoryPaidoussis1966a,GregoryPaidoussis1966b,SteinTobriner1970,MR1727216,TYLIKOWSKI1979141,Matviichuk1990,MR1449838,MR1792245,MR1274332} to name just a few, and their results -- analytical, numerical, and experimental -- are very well known by now. For example, the classic result of Movchan \cite{Movchan1965} is that while solutions of the \textbf{OLS} are stable in the interval $0\leq \eta<\pi$, they are unstable for $\eta=n\pi$, $n\in\mathbf{N}$. We refer the interested reader to the monograph of Pa\"idoussis \cite{Paidoussis2014} for further extensive information and historical references on the stability of the \textbf{OLS} and its many variants, along with numerous experimental results. A comprehensive survey of the earlier work is also available in the book by Thompson \cite{Thompson1982}.

In the mathematical literature, the first operator-theoretic or abstract account of the \textbf{OLS} goes back to Holmes and Marsden \cite{MR495662} who studied the stability problem for a more general class of systems related to equations occurring in panel flutter problems, including the \textbf{OLS} as a special case, using Liapunov's direct method in the framework of \textit{operator semigroups} in Hilbert space (see also \cite[Section 7.5]{MR1262126}). Shortly afterwards Röh \cite{Roh1982} and, independently, Miloslavskii \cite{Miloslavskii1985,Miloslavskii1983} rigourised the spectral approach used in the engineering literature within the semigroup framework to obtain stability results for the \textbf{OLS}. Their works can in fact rightly be regarded as somewhat a milestone in the spectral approach to the stability analysis of general vibrating systems governed by partial differential equations. (Earlier steps in this direction can be found in \cite{MR0478948,MR0585481,MR0543945,MR0355650,MR0355650s}, where in the papers by Walker and Infante \cite{MR0478948} and Carr and Malhardeen \cite{MR0585481,MR0543945} it was motivated by and applied to the study of stability of Beck's Problem.)

Independently, a rather complete development of the spectral theory of \textit{quadratic operator pencils} and its applications to the stability analysis of the \textbf{OLS} and its variants was given in a series of papers notably by Adamyan and Pivovarchik \cite{MR1727987}, Artamonov \cite{Artamonov2000}, Lyong \cite{Lyong1993}, Miloslavskii \cite{Miloslavskii1981,Miloslavskii1991}, Pivovarchik \cite{Pivovarchik1993,Pivovarchik1994,Pivovarchik1992}, Shkalikov \cite{Shkalikov1996}, and Zefirov et al.\ \cite{MiloslavskiiEtAl1985}. In particular, the possibility first raised in \cite{PaidoussisIssid1974} of having regions in the $\eta\beta$-plane for stability recovery, or so-called gyroscopic stabilisation, meaning that although solutions of the \textbf{OLS} are unstable for some $\eta$ and a critical value of $\beta$, they can regain stability for the same $\eta$ and some $\beta$ above the critical value, was explored in \cite{MiloslavskiiEtAl1985}. It was found that when $\eta=2\pi+\epsilon^2$ (small $\epsilon$), the \textbf{OLS} no longer is unstable for values of $\beta$ above $3^{-{1}/{2}}$ -- i.e., stability recovery takes place. However, as stated in that paper, stability recovery is not possible in the intervals $\left(2n-1\right)\pi\leq \eta\leq2n\pi$, $n\in\mathbf{N}$, regardless of the values of $\beta$. The subject of computing estimates for the \textbf{OLS} of the size of the stability regions in the $\eta\beta$-plane was also extensively studied in the above-cited works in \cite{Pivovarchik1993,Pivovarchik1994,Pivovarchik2005} (see also \cite[Chapter 4]{MollerPivovarchik2015}).

We return to the stability problem for the $\mathbf{CLS}$ which has not been addressed so far. Unfortunately, in general, the requirement of stability places rather severe restrictions on the admissibility of damping when a system is circulatory-gyroscopic: If there is flow then arbitrarily small damping can lead to destabilisation. This phenomenon -- sometimes referred to in the literature as ``destabilisation paradox'' -- has been known at an abstract level (generalising the classic Kelvin--Tait--Cetaev theorem) for a relatively long time, see \cite{Miloslavskii1991,Pivovarchik1992,Shkalikov1996} and the references therein. So in the case of the $\mathbf{CLS}$ we may expect that when there is flow, $\eta>0$, stability can be destroyed by the presence of boundary damping represented by $\kappa>0$, no matter how small. We will show this is not the case here. In fact, we prove that under the condition that $\gamma> \eta^2$, which we will from now assume apply, solutions of the $\mathbf{CLS}$ are exponentially stable when boundary damping is present. The result is interesting but perhaps not surprising from physical considerations; yet, we are unaware of any mathematical proof in the literature that deduces the exponential stability of the $\mathbf{CLS}$ as a consequence of the spectral approach taken here. In this respect, it is an interesting observation that if we replace the boundary conditions \eqref{eq_06} by
\begin{equation*}
\left.\frac{\partial^2 \bm{w}\left(s,t\right)}{\partial s^2}\right|_{s=1}=\bm{u}\left(t\right),\quad \left.\frac{\partial^3 \bm{w}\left(s,t\right)}{\partial s ^3}\right|_{s =1}= 0,
\end{equation*}
then the tube can never be exponentially stabilised by the boundary controller \eqref{eq_06ssyy} in conjunction with the other boundary conditions \eqref{eq_02b}, because, in this case, zero is an eigenvalue (see Remark \ref{remarkxdd123}). In fact, in this case, the system is not even stable.

Where appropriate in this paper, we examine the connection between the two abstract formulations mentioned in the review above -- semigroup formulation and quadratic operator-pencil formulation -- as needed to treat the stability problem for the $\mathbf{CLS}$ without restriction to a specific abstract framework for the boundary-eigenvalue problem associated with it. Such examinations, in general, require careful study of both the spectral properties, i.e.\ existence, location, multiplicities, and, in particular, asymptotics of eigenvalues, and of the basis properties or basisness for root vectors (eigen- and associated vectors or chains of eigen- and associated vectors, see Definition \ref{def01}) of the operators involved. For this purpose, let us begin with the ``standard'' boundary-eigenvalue problem associated with the $\mathbf{CLS}$. Anticipating the results of the paper, we are fully justified in considering a separable solution of the form $\bm{w}\left(\,\cdot\,,t\right)=w\exp\left(\lambda t\right)$ to obtain the following boundary-eigenvalue problem for $w$ with spectral parameter $\lambda$:
\begin{equation}\label{eq_1sffss1}
\left\{\begin{split}
w^{(4)}-(\gamma-\eta^2)\,w''+2\lambda\beta\eta w'+\lambda^2w&=0,\\
w\left(0\right)=w''\left(0\right)&=0,\\
w''\left(1\right)+\lambda\kappa w'\left(1\right)&=0,\\
w^{(3)}\left(1\right)-(\gamma-\eta^2)\, w'\left(1\right)&=0,
\end{split}\right.
\end{equation}
where the prime denotes differentiation with respect to $s$, as usual. It can be checked that with the spectral parameter transformation $\lambda\mapsto i\rho^2$, the boundary-eigenvalue problem \eqref{eq_1sffss1} is Birkhoff regular in the sense of \cite[Definition 7.3.1]{MennickenMoller2003}.

The corresponding spectral problem applying to a semigroup formulation in the energy space setting is obtained on the basis of a linearisation (in the spectral parameter) of \eqref{eq_1sffss1}. That is, by setting $v=\lambda w$ one obtains the boundary-eigenvalue problem
\begin{equation}\label{eq_1sffss1ss}
\left\{\begin{split}
v&=\lambda w,\\
w^{(4)}-(\gamma-\eta^2)\,w''+2\beta\eta v'+\lambda v&=0,\\
w\left(0\right)=w''\left(0\right)&=0,\\
w''\left(1\right)+\kappa v'\left(1\right)&=0,\\
w^{(3)}\left(1\right)-(\gamma-\eta^2)\, w'\left(1\right)&=0.
\end{split}\right.
\end{equation}

\subsection{Abstract formulation}

We now reformulate boundary-eigenvalue problems \eqref{eq_1sffss1} and \eqref{eq_1sffss1ss} abstractly in appropriate Hilbert spaces. Since \eqref{eq_1sffss1} has $\lambda$-dependent boundary conditions, it is impossible to recast it abstractly as a spectral problem for linear operators in $\bm{L}_2$ (as was done for the special case of the \textbf{OLS} in the papers cited previously). However, it gives rise to the spectral problem for the operator pencil
\begin{equation}\label{eq_11}
\mathcal{M}\left(\lambda\right)=\lambda^2G+\lambda D+C,\quad \lambda\in\mathbf{C},
\end{equation}
in the Hilbert product space
\begin{equation*}
\mathbb{Y}=\bm{L}_2\left(0,1\right)\times\mathbf{C}
\end{equation*}
with the induced inner product and norm denoted by $\left<\,\cdot\,,\,\cdot\,\right>$ and $\left\|\,\cdot\,\right\| $, respectively; here
\begin{equation*}
\left<y,\tilde{y}\right>=\int_0^1w\left(s\right)\overline{\tilde{w}\left(s\right)}ds+
c\overline{\tilde{c}},\quad \left\| y\right\|=\left<y,{y}\right>^{{1}/{2}}.
\end{equation*}
The operators $C$, $D$, $G$ appearing in \eqref{eq_11} have domains
\begingroup
\allowdisplaybreaks
\begin{gather*}
\bm{D}\left({C}\right)=\left\{y=\left(\begin{matrix}
w\\
c
\end{matrix}\right)\in\mathbb{Y}~\middle|
~\begin{gathered}
w\in {\bm{H}}^4\left(0,1\right), \\
 c=w'\left(1\right),~w\left(0\right)=w''\left(0\right)=0,~
w^{(3)}\left(1\right)-(\gamma-\eta^2)\, w'\left(1\right)=0
\end{gathered}\right\},\\
\bm{D}\left({D}\right)=\bm{D}\left({C}\right),\quad \bm{D}\left({G}\right)=\mathbb{Y}
\end{gather*}
\endgroup
and are defined by
\begin{equation*}
C{y}\coloneqq\left(\begin{matrix}
w^{(4)}-(\gamma-\eta^2)\, w''\\
w''\left(1\right)
\end{matrix}
\right),\quad D{y}\coloneqq\left(\begin{matrix}
2\beta\eta w'\\
\kappa w'\left(1\right)
\end{matrix}
\right),\quad 
Gy\coloneqq\left(\begin{matrix}
w\\
0
\end{matrix}
\right).
\end{equation*}
Throughout this paper, we will use the standard notation $\bm{H}^k\left(0,1\right)$, $k\in\mathbf{N}_0$, for the Sobolev--Hilbert space of order $k$ associated with $\bm{L}_2\left(0,1\right)$. We will also use the same symbols $\left<\,\cdot\,,\,\cdot\,\right>$ and $\left\|\,\cdot\,\right\| $ to denote, respectively, any one of the inner products and norms when it is perfectly clear from the usage which one is intended, but we reserve $\left\|\,\cdot\,\right\|_0 $ and $\left<\,\cdot\,,\,\cdot\,\right>_0$ for the usual $\bm{L}_2$ norm and inner product.

Using \cite[Theorems 10.3.5 and 10.3.8]{MollerPivovarchik2015}, one can verify that (i) $C=C^\ast\gg0$ and $C^{-1}$ is compact; (ii) $D$ is completely subordinate to $C$ or, here equivalently, $C$-compact (in the sense of \cite[Section IV.1.3]{Kato1995}); (iii) $D\ge 0$ and is of rank 1 for $\eta=0$; and (iv) $G\ge 0$ and is bounded. By definition, the domain of $\mathcal{M}\left(\lambda\right)$ is given by $\bm{D}\left(\mathcal{M}\left(\lambda\right)\right)=\bm{D}\left(C\right)\cap\bm{D}\left(D\right)\cap\bm{D}\left(G\right) =\bm{D}\left(C\right)$ and thus is $\lambda$-independent. For every $y\in\bm{D}\left(C\right)$, there holds $\mathcal{M}\left(\lambda\right)y=0$ if and only if \eqref{eq_1sffss1} holds. So $\mathcal{M}\left(\lambda\right)$ represents the boundary-eigenvalue problem \eqref{eq_1sffss1}. Since \eqref{eq_1sffss1} with $\lambda$ replaced by $i\rho^2$ is Birkhoff regular, the spectrum of the pencil $\mathcal{M}$ consists of an infinite number of eigenvalues of finite type or normal eigenvalues. Indeed, since $\mathcal{M}\left(\lambda\right)$ is a relatively compact perturbation of the Fredholm operator $\mathcal{M}\left(0\right)=C$, we have, by \cite[Theorem IV.5.26]{Kato1995}, that $\mathcal{M}\left({\lambda}\right)$ has a compact resolvent. (Note that the assumption $\gamma> \eta^2$ is necessary for the assertion (i) to hold.) We recall the following standard notions from the spectral theory of operator pencils in a Hilbert space (e.g.\ see \cite[Definitions 1.1.2 and 1.1.3]{MollerPivovarchik2015}).
\begin{definition}\label{def01}
Let $\lambda\mapsto \mathcal{L}\left({\lambda}\right)$ be a mapping from $\mathbf{C}$ into the set of closed linear operators in a Hilbert space. The resolvent set of $\mathcal{L}$, denoted by $\varrho\left(\mathcal{L}\right)$, is the set of $\lambda$ for which $\mathcal{L}\left({\lambda}\right)$ is boundedly invertible (i.e.\ the inverse $\mathcal{L}\left({\lambda}\right)^{-1}$ exists, is closed and bounded). We call $\mathcal{L}\left({\lambda}\right)^{-1}$ the resolvent of $\mathcal{L}\left({\lambda}\right)$. The spectrum of $\mathcal{L}$ is the set of $\lambda\not\in\varrho\left(\mathcal{L}\right)$ and is denoted by $\sigma\left(\mathcal{L}\right)$. If a number $\lambda\in\mathbf{C}$ has the property that $\operatorname{ker}\mathcal{L}\left({\lambda}\right)\neq\left\{0\right\}$
then it is called an eigenvalue of $\mathcal{L}$ and there exists an eigenvector $x\neq 0$ corresponding to $\lambda$ such that $\mathcal{L}\left({\lambda}\right)x=0$. The vectors $x_0,x_1,\ldots,x_{m-1}$ are said to form a chain, of length $m$, consisting of an eigenvector $x_0$ of $\mathcal{L}$ corresponding to an eigenvalue $\lambda_0$ and the vectors $x_1,x_2,\ldots,x_{m -1}$ associated with it, or simply a chain of root vectors of $\mathcal{L}$ corresponding to $\lambda_0$, if
\begin{equation}\nonumber
\sum^j_{k=0}\left.\frac{1}{k!}\frac{d^k}{d\lambda^k}\mathcal{L}\left(\lambda\right)\right|_{\lambda=\lambda_0}x_{j-k}=0,\quad j=0,1,\ldots,m -1.
\end{equation}
The geometric multiplicity of an eigenvalue ${\lambda}_0$ is the number of linearly independent eigenvectors in a system of chains of root vectors of $\mathcal{L}$ corresponding to $\lambda_0$ and is defined as $\operatorname{dim}\operatorname{ker}\mathcal{L}\left({\lambda}_0\right)$. The algebraic multiplicity of an eigenvalue $\lambda_0$ is the maximum value of the sum of the lengths of chains corresponding to the linearly independent eigenvectors and is denoted by ${\lambda}_0$. We call an eigenvalue ${\lambda}_0$ semisimple if its geometric and algebraic multiplicities are equal, $\nu\left(\lambda_0\right)=\operatorname{dim}\operatorname{ker}\mathcal{L}\left({\lambda}_0\right)$, and simple if $\nu\left(\lambda_0\right)=\operatorname{dim}\operatorname{ker}\mathcal{L}\left({\lambda}_0\right)=1$. If an eigenvalue ${\lambda_0}$ is an isolated point in $\sigma\left(\mathcal{L}\right)$ and $\mathcal{L}\left({\lambda}_0\right)$ is a Fredholm operator, then we call $\lambda_0$ a normal eigenvalue. The set of all normal eigenvalues is denoted by $\sigma_0\left(\mathcal{L}\right)$, which forms the discrete spectrum
of $\mathcal{L}$.
\end{definition}


We now introduce the $\mathbf{CLS}$ operator and the space in which it will be considered. Let
\begin{equation*}
\mathbb{X}= \mathring{\bm{H}}^2\left(0,1\right)\times \bm{L}_2\left(0,1\right),
\end{equation*}
where here and in the associated operator domain definitions below
\begin{equation*}
\mathring{\bm{H}}^k\left(0,1\right)\coloneqq\left\{w\in {\bm{H}}^k\left(0,1\right)~\middle|~w\left(0\right)=0\right\},\quad k=1,2.
\end{equation*}
The space $\mathbb{X}$ is a closed subspace of ${\bm{H}}^2\left(0,1\right)\times \bm{L}_2\left(0,1\right)$ which we will call the energy space or \textit{state space}. Let us define the norm of an element ${x}\in \mathbb{X}$ as the norm induced by the inner product $\left<\,\cdot\,,\,\cdot\,\right>$ in $\mathbb{X}$, i.e.\
\begin{equation}\label{eq_11abc}
\left<x,\tilde{x}\right>=\int^1_0\left[w''\left(s\right)\overline{\tilde{w}''\left(s\right)}+(\gamma-\eta^2)\, w'\left(s\right)\overline{\tilde{w}'\left(s\right)}+v\left(s\right)\overline{\tilde{v}\left(s\right)}\right]ds,\quad
\left\|x\right\|= \left<x,{x}\right>^{{1}/{2}},
\end{equation}
a form equivalent to the square root of the energy (cf.\ \eqref{eqaas345}). With this norm
$\mathbb{X}$ becomes a Hilbert space.

In $\mathbb{X}$ we introduce the $\mathbf{CLS}$ operator
\begin{equation}\label{eq191922}
T= A+B
\end{equation}
with
\begin{equation}\label{eq_09xa}
A{x}\coloneqq\left(\begin{matrix}
v\\
-w^{(4)}+(\gamma-\eta^2)\, w''
\end{matrix}\right), \quad B{x}\coloneqq\left(\begin{matrix}
0\\
-2\beta\eta v'\
\end{matrix}\right),
\end{equation}
where $A$ has domain
\begin{equation}\label{eq_09xb}
\bm{D}\left({A}\right)=\left\{x=\left(\begin{matrix}
w\\
v
\end{matrix}\right)\in\mathbb{X}~\middle|
~\begin{gathered}
w\in{\bm{H}}^4\left(0,1\right)\cap \mathring{\bm{H}}^2\left(0,1\right),~v\in\mathring{\bm{H}}^2\left(0,1\right), \\
w''\left(0\right)=0,~ w''\left(1\right)+\kappa v'\left(1\right)=0,\\
w^{(3)}\left(1\right)-(\gamma-\eta^2)\, w'\left(1\right)=0
\end{gathered}\right\},
\end{equation}
while $B$ has domain
\begin{equation}\label{eq_09xc}
\bm{D}\left({B}\right)=\left\{x=\left(\begin{matrix}
w\\
v
\end{matrix}\right)\in\mathbb{X}~\middle|
~ w\in\mathring{\bm{H}}^2\left(0,1\right),~v\in\mathring{\bm{H}}^1\left(0,1\right)\right\}.
\end{equation}
Evidently, by definition,
\begin{equation*}
\bm{D}\left({T}\right)=\bm{D}\left({A}\right).
\end{equation*} 
The linearised boundary-eigenvalue problem \eqref{eq_1sffss1ss} is then verified at once to be equivalent to the spectral problem for the linear pencil
\begin{equation}\label{eq1lipenc12}
\mathcal{P}\left(\lambda\right)=\lambda I-T,\quad \lambda\in\mathbf{C},
\end{equation}
in the state space $\mathbb{X}$.

It is easily seen that a chain of root vectors of the pencil $\mathcal{P}$ coincides with a chain of root vectors of $T$, corresponding to the same eigenvalue. Therefore, the eigenvalues of $T$ coincide, including multiplicities, with those of the pencils $\mathcal{P}$ and $\mathcal{M}$ and are the same as those from the boundary-eigenvalue problem \eqref{eq_1sffss1}. Moreover, a chain of root functions $w_0,w_1,\ldots,w_{m-1}$ of \eqref{eq_1sffss1} corresponds to a chain of root vectors $y_0,y_1,\ldots,y_{m-1}$ of the pencil $\mathcal{M}$ ($m$, we recall, the length of the chains).
If we were now to consider the problem of basisness in the space $\mathbb{Y}$ we would note that it was connected to that in $\bm{L}_2\left(0,1\right)$ and would begin with consideration of the root functions. So the question of obvious interest is as to what may be inferred from basisness of the root functions as far as what the basisness of a cahin of root vectors $x_0,x_1,\ldots,x_{m-1}$ of $T$ in $\mathbb{X}$ might be. We give a careful analysis of this matter in Section \ref{sec_4}.

\subsection{Well-posedness}

As will be seen, the $\mathbf{CLS}$ operator $T$ is maximal dissipative with compact resolvent (i.e., $\mathcal{P}\left({\lambda}\right)^{-1}$ is compact for some and thus for all $\lambda\in\varrho\left(\mathcal{P}\right)$). We can then invoke the familiar Lumer--Phillips theorem from the theory of strongly continuous semigroups of operators or $C_0$-semigroups (see, e.g., \cite[Section II.3.b]{EngelNagel1999}, \cite[Section 1.4]{Pazy1983} or \cite[Section I.4.2]{Krein1971}) to show that writing the $\mathbf{CLS}$ in the form
\begin{equation}\label{eq_03xab}
\dot{{x}}\left(t\right)=Tx\left(t\right),\quad {{x}}\left(t\right)=\left(\begin{matrix}
\bm{w}\left(\,\cdot\,,t\right)\\
\bm{v}\left(\,\cdot\,,t\right)
\end{matrix}\right),\quad {{x}}\left(0\right)=x_0=\left(\begin{matrix}
g\left(\,\cdot\,\right)\\
h\left(\,\cdot\,\right)
\end{matrix}\right)
\end{equation}
gives the solution
\begin{equation}\label{eq_03}
{{x}}\left(t\right)=\mathcal{U}\left(t\right)x_0,\quad t\ge0,
\end{equation}
for any suitably smooth $x_0$, e.g., for $x_0\in \bm{D}\left(A\right)$. Here $\mathcal{U}\left(t\right)\left(\coloneqq\exp\left(tT\right)\right)$ is a contraction $C_0$-semigroup on $\mathbb{X}$ (embeds in a $C_0$-group, actually, as will be indicated in Section \ref{sec_4b}), with infinitesimal generator $T$. This holds if and only if $\bm{w}\left(s,t\right)$ is a solution of the $\mathbf{CLS}$ and
$\bm{v}\left(s,t\right)=\left(\partial \bm{w}/\partial t\right)\left(s,t\right)$. Then, clearly (with norm as given in \eqref{eq_11abc})
\begin{equation*}
\mathbf{E}\left(t\right)=\frac{1}{2}\left\|{{x}}\left(t\right)\right\|^2.
\end{equation*}
Therefore, when dealing with our programme of analysing the asymptotic behaviour of the energy of solutions of the $\mathbf{CLS}$ in time, we can study the asymptotic behaviour of $\left\|x\left(t\right)\right\|$ as $t\rightarrow\infty$.

The following theorem disposes of the question of well-posedness of the initial-value problem \eqref{eq_03xab}, hence of the $\mathbf{CLS}$, guaranteeing that a unique solution of the form \eqref{eq_03} exists for any $x_0\in \bm{D}\left(A\right)$.
\begin{theorem}\label{L-2-1}
The $\mathbf{CLS}$ operator $T$ has a compact inverse, is maximal dissipative for $\eta,\kappa>0$ and skewadjoint for $\eta=\kappa=0$. Thus $T$ is for $\eta,\kappa\geq 0$ the infinitesimal generator of a contraction $C_0$-semigroup on $\mathbb{X}$.
\end{theorem}
\begin{proof}
For the first statement consider the equation $Ax=\tilde{x}$ with $\tilde{x}\in \mathbb{X}$ and $x\in \bm{D}\left(A\right)$. Equivalently
\begin{equation}\label{eq_1sffss1xxx}
\left\{\begin{split}
v&=\tilde{w},\\
-w^{(4)}+(\gamma-\eta^2)\, w''&=\tilde{v},\\
w\left(0\right)=w''\left(0\right)&=0,\\
w''\left(1\right)+\kappa v'\left(1\right)&=0,\\
w^{(3)}\left(1\right)-(\gamma-\eta^2)\, w'\left(1\right)&=0.
\end{split}\right.
\end{equation}
Integrating the differential equation in \eqref{eq_1sffss1xxx} twice from $0$ to $1$, making use of the boundary conditions $w\left(0\right)=w''\left(0\right)=0$ and $w^{(3)}\left(1\right)-(\gamma-\eta^2)\, w'\left(1\right)=0$, we get
\begin{equation*}
w''\left(s\right)-(\gamma-\eta^2)\, w\left(s\right)=\int_0^sdt\int_t^1\tilde{v}\left(r\right)dr.
\end{equation*}
So
\begin{equation*}
w\left(s\right)=a\sinh\sqrt{\gamma-\eta^2}s+\frac{1}{\sqrt{\gamma-\eta^2}}\int_0^s\sinh\sqrt{\gamma-\eta^2}\left(s-r\right)\tilde{V}\left(r\right)dr,
\end{equation*}
arbitrary constant $a$, with the integral term $\tilde{V}\left(s\right)=\int^s_0dt\int^1_t\tilde{v}\left(r\right)dr$ is the solution of the differential equation satisfying the aforementioned three boundary conditions. Application of the remaining boundary condition gives
\begin{equation*}
a=\frac{-\sqrt{\gamma-\eta^2}\int_0^1\sinh\sqrt{\gamma-\eta^2}\left(1-r\right)\tilde{V}\left(r\right)dr-\tilde{V}\left(1\right)-\kappa\tilde{w}'\left(1\right)}{(\gamma-\eta^2)\,\sinh\sqrt{\gamma-\eta^2}}\coloneqq b\left(\tilde{w},\tilde{v}\right),
\end{equation*}
where we have used that $v=\tilde{w}$. Thus the inverse operator
\begin{equation*}
(A^{-1}\tilde{x})\left(s\right)=\left(\begin{matrix}
 b\left(\tilde{w},\tilde{v}\right)\sinh\sqrt{\gamma-\eta^2}s+\frac{1}{\sqrt{\gamma-\eta^2}}\int_0^s\sinh\sqrt{\gamma-\eta^2}\left(s-r\right)\tilde{V}\left(r\right)dr\\
\tilde{w}\left(s\right)
\end{matrix}\right)
\end{equation*}
exists, $A^{-1}\tilde{x}=x\in \bm{D}\left(A\right)$ being the unique solution of $Ax=\tilde{x}$, and this together with the closed graph theorem yields that $A^{-1}$ is a closed and bounded operator; hence, $A$ is boundedly invertible. Since $\bm{D}\left(A\right)\subset \mathbb{X}$, $A^{-1}$ is compact via Sobolev's embeddings ${\bm{H}}^4\left(0,1\right)\cap \mathring{\bm{H}}^2\left(0,1\right)\xhookrightarrow{} \mathring{\bm{H}}^2\left(0,1\right) \xhookrightarrow{} \bm{L}_2\left(0,1\right)$. Furthermore, we have that
\begin{equation*}
BA^{-1}\left(\begin{matrix}
\tilde{w}\\
\tilde{v}
\end{matrix}\right) =\left(\begin{matrix}
0\\
-2\beta\eta\tilde{w}'
\end{matrix}\right)
\end{equation*}
and so, $\tilde{w}$ being an element of $\mathring{\bm{H}}^2\left(0,1\right)$, we have that $BA^{-1}$ is compact. Using the compactness of $A^{-1}$ together with that of $BA^{-1}$ we see that $T$ is a relatively compact perturbation of the  Fredholm operator $A$. Therefore $T^{-1}$ is compact (see \cite[Theorem IV.5.26]{Kato1995}).

Now let us verify dissipativity of $T$. For any $x\in \bm{D}\left(A\right)$, we compute without difficulty
\begin{equation}\label{eq_13}
2\operatorname{Re}\left<Tx,x\right>=\left<Tx,x\right>+\left<x,Tx\right>=-2\beta\eta\left|v\left(1\right)\right|^2-2\kappa\,|v'\left(1\right)\!|^2,
\end{equation}
and that $\operatorname{Re}\left<Tx,x\right>\le 0$ is clear using that $\beta\in\left(0,1\right)$ and $\eta,\kappa\geq0$. We then conclude that $T$ is a closed dissipative operator and hence via the previous arguments is maximal dissipative. Indeed the previous arguments show that $0\in\varrho\left(T\right)$ and therefore since $\lambda I-T$ is onto for sufficiently small $\lambda>0$ its range coincides with $\mathbb{X}$ (which is in turn a simple consequence of the contraction fixed point theorem, see, e.g., \cite[Proposition 7.1]{MR2759829}). From \eqref{eq_13} it is clear that $T$ is skewsymmetric for $\eta=\kappa=0$, and the skewadjointness follows from the surjectivity.

The second statement follows from the first by the Lumer--Phillips theorem because in Hilbert space a closed, maximal dissipative operator is densely defined (see \cite[Proposition 7.1]{MR2759829} or \cite[Section V.3.10]{Kato1995}).
\end{proof}


\begin{remark}\label{remarkxdd123}
We note that when in \eqref{eq_1sffss1} the boundary condition $w^{(3)}\left(1\right)-(\gamma-\eta^2)\, w'\left(1\right)=0$  is replaced by $w^{(3)}\left(1\right)=0$ we have the so-called degenerate case where $\lambda=0$ is an eigenvalue. Let us quickly verify this. Clearly any nontrivial solution of the differential equation in \eqref{eq_1sffss1} in this case can be written explicitly in the form
\begin{equation*}
w\left(s\right)=a_1+a_2s+a_3\sinh\sqrt{\gamma-\eta^2}s+a_4\cosh\sqrt{\gamma-\eta^2}s
\end{equation*}
for some arbitrary coefficients $a_r$, $r=1,2,3,4$. Using the boundary conditions at $s=0$ we have $a_1=a_4=0$. If $w''\left(1\right)=0$ and $w^{(3)}\left(1\right)-(\gamma-\eta^2)\, w'\left(1\right)=0$, we obtain $a_2=a_3=0$. This gives the trivial solution $w\left(s\right)=0$ and hence a contradiction. Now suppose $w^{(3)}\left(1\right)=0$ instead. Then $a_3=0$ and so $\lambda=0$ is an eigenvalue. The corresponding eigenfunction is given by $w\left(\lambda,s\right)=s$ (up to a multiplicative constant).
\end{remark}

\subsection{Characterisation of stability and exponential stability}\label{sec2_3}

We need some basic facts from stability of $C_0$-semigroups which can be found in any standard text (such as \cite[Chapter 5]{CurtainZwart1995} or \cite[Chapter V]{EngelNagel1999}).

The basic definition of stability of the semigroup $\mathcal{U}\left(t\right)$ is that for some $\varepsilon\ge 0$ there should exist a constant $M\ge 1$ such that $\left\|\mathcal{U}\left(t\right)\right\| \leq M e^{-\varepsilon t}$, $t\ge 0$, and hence that, for any $x_0\in \bm{D}\left(A\right)$,
\begin{equation}\label{eq_05}
\left\|\mathcal{U}\left(t\right)x_0\right\| \leq M e^{-\varepsilon t}\left\|x_0\right\|,\quad t\ge 0.
\end{equation}
Then solutions of the $\mathbf{CLS}$ are called stable or bounded stable for $\varepsilon=0$ (i.e., the contractive case for $M=1$) and exponentially stable when $\varepsilon>0$. In the latter situation, we say that $\mathcal{U}\left(t\right)$ is exponentially stable with energy decay rate $\varepsilon$. (We wish to point out that in the papers using an abstract description of the $\mathbf{OLS}$, e.g.\ those by Milsolavskii and Röh cited at the beginning of the section, \eqref{eq_05} is considered with respect to the same energy or $\mathbb{X}$-topology, where the space $\mathring{\bm{H}}^2\left(0,1\right)$ in the definition of the state space $\mathbb{X}$ is $\mathring{\bm{H}}^2\left(0,1\right)\coloneqq\left\{w\in {\bm{H}}^2\left(0,1\right)~\middle|~w\left(0\right)=w'\left(0\right)=0\right\}$, but a norm different from ours in \eqref{eq_11abc} is used, namely,
\begin{equation*}
\left\|x\right\|'= \left[\int^1_0\left(\left|w''\left(s\right)\right|^2+\left|v\left(s\right)\right|^2\right)ds\right]^{{1}/{2}}.
\end{equation*}
In the cantilever case, it is not difficult to see that $\left\|\,\cdot\,\right\|'$ will be a norm but, in our case, it is not.)

In \cite{MiloslavskiiEtAl1985}, it is shown that for stability it is necessary and sufficient that the eigenvalues be purely imaginary and semisimple, the latter meaning that there be only corresponding eigenvectors and no associated vectors (see Definition \ref{def01}). So in the terminology of \cite{MiloslavskiiEtAl1985}, in the \textbf{CLS} case, by stability of the pencil $\mathcal{M}$ is meant stability of the $\mathbf{CLS}$, or equivalently of the semigroup generated by the $\mathbf{CLS}$ operator $T$. However, special technical difficulties arise when it comes to stating conditions for the $\mathbf{CLS}$ to be exponentially stable. To see this let there exist $\varepsilon>0$. We first recall that exponential stability is equivalent to the condition that the type 
\begin{equation*}
\omega\left(T\right)\coloneqq\lim_{t\rightarrow\infty}\frac{1}{t}\log\left\|\mathcal{U}\left(t\right)\right\|\leq -\varepsilon;
\end{equation*}
but while there always holds, by the Hille--Yoshida theorem,
\begin{equation}\label{eq05xav}
\omega\left(T\right)\geq \sup\left\{\operatorname{Re}\left(\lambda\right)~\middle|~\lambda\in\sigma\left(T\right)\right\},
\end{equation}
equality in \eqref{eq05xav} unfortunately does not necessarily hold if $T$ is an unbounded operator and so, equivalently, the spectral mapping property $\sigma\left(\mathcal{U}\left(t\right)\right)\backslash\left\{0\right\}=\exp\left(t\sigma\left(T\right)\right)$, $t\ge 0$, is not valid in general. It is a consequence, therefore, of the aforementioned characterisation that it is not generally justified to require for exponential stability only that the spectrum of $T$ be confined to the open left half of the complex plane, in the sense that $\sup\left\{\operatorname{Re}\left(\lambda\right)~\middle|~\lambda\in\sigma\left(T\right)\right\}<0$ and hence that
\begin{equation}\label{eq05x}
\sup\left\{\operatorname{Re}\left(\lambda\right)~\middle|~\lambda\in\sigma\left(T\right)\right\}\leq -\varepsilon<0.
\end{equation}

There are conditions for the inequality \eqref{eq05xav} to allow for equality, in which case one says that $T$ satisfies the \textit{spectrum-determined growth} assumption or condition (a term introduced by Triggiani in \cite{Triggiani1975}). For example, equality holds for the large class of eventually norm-continuous semigroups, including the special cases of eventually differentiable semigroups -- i.e.\ $\left\|\,\cdot\,\right\|$-differentiable -- and compact semigroups (see \cite[Corollary IV.3.12]{EngelNagel1999}). A condition of particular interest to us, however, is formulated in terms of the unconditional or Riesz basisness (i.e.\ basisness equivalent to orthonormal basisness) of the root vectors, which can be deduced readily from a theorem of Miloslavskii \cite[Theorem 1]{Miloslavskii1985}. Indeed, concerning Riesz basisness, we have the following elementary argument: Let $\left\{\lambda_n\in\mathbf{C}\right\}$ be a properly enumerated sequence of simple eigenvalues of the $\mathbf{CLS}$ operator $T$ and suppose there is a corresponding sequence of eigenvectors $\left\{x_n\right\}$ of $T$ which forms a Riesz basis for ${\mathbb{X}}$.
There exists then a unique sequence of vectors $\left\{z_n\right\}$ such that the pairs $\left\{x_n,z_n\right\}$ form a biorthogonal sequence in ${\mathbb{X}}$. Since every $x_0\in{\mathbb{X}}$ has a norm-convergent expansion in the Riesz basis of eigenvectors, which we assume has the series form
\begin{equation*}
x_0=\sum^\infty_{n=1}\left<x_0,z_n\right> x_n,
\end{equation*}
the solution \eqref{eq_03} can be represented as
\begin{equation*}
\mathcal{U}\left(t\right)x_0=\sum^\infty_{n=1}e^{\lambda_nt}\left<x_0,z_n\right> x_n,\quad t\ge 0.
\end{equation*}
Assume that \eqref{eq05x} holds. By using (cf.\ \cite[Section VI.2.(2.4)]{GohbergKrein1969})
\begin{equation*}
M_1\sum^\infty_{n=1}\left|\left<x_0,z_n\right> \right|^2\leq\left\|x_0\right\|^2 \leq M_2\sum^\infty_{n=1}\left|\left<x_0,z_n\right> \right|^2
\end{equation*}
for some constants $M_1,M_2> 0$ we estimate for $t\ge 0$
\begin{equation*}
\left\|\mathcal{U}\left(t\right)x_0\right\|^2 \leq M_2\sum^\infty_{n=1}\left|e^{\lambda_nt}\left<x_0,z_n\right> \right|^2\leq M_2e^{-2\varepsilon t}\sum^\infty_{n=1}\left|\left<x_0,z_n\right> \right|^2\leq\frac{M_2}{M_1}e^{-2\varepsilon t}\left\|x_0\right\|^2 ,
\end{equation*}
which is essentially \eqref{eq_05}.

As we have just remarked, stability of the $\mathbf{CLS}$ is really the same as stability of the pencil $\mathcal{M}$. The question remains: Can we conclude exponential stability of the $\mathbf{CLS}$ from that of the pencil $\mathcal{M}$; and if so, what does exponential stability of the pencil $\mathcal{M}$ actually mean? Ultimately, we would like to know if the root vectors of $T$ also form a Riesz basis for ${\mathbb{X}}$, when those of the pencil $\mathcal{M}$ do for ${\mathbb{Y}}$. This appears to be a worthwhile question in its own right, the answer to which is not at all obvious. In this paper we will provide at least a partial answer to the question and study the problem of basisness in $\bm{L}_2\left(0,1\right)$.

The organisation of the rest of the paper is as follows. In Section \ref{sec_3} we analyse the spectral properties of the $\mathbf{CLS}$ operator $T$. We use the eigenvalue asymptotics in Section \ref{sec_4} to prove that there exists a corresponding sequence of eigenfunctions of the boundary-eigenvalue problem \eqref{eq_1sffss1} forming a Riesz basis for $\bm{L}_2\left(0,1\right)$. We then prove that the sequence of eigenvectors of $T$ have the property of being a Riesz basis for the state space $\mathbb{X}$. The proofs are direct and based on several verifications that the sequences of eigenfunctions and eigenvectors are quadratically close to some orthonormal bases for the spaces considered. They are then completed by appeal to a well-known theorem due to Bari (see \cite[Chapter VI]{GohbergKrein1969}, \cite[Section V.2.5]{Kato1995} or \cite[Chapter 1]{Young1980} for details). As will be seen in Section \ref{sec_3}, the restriction $\gamma> \eta^2$ imposed on the tensile forces acting on the tube guarantees that, for $\eta,\kappa>0$, the eigenvalues approach a vertical asymptote in the left half-plane a finite distance away from the imaginary axis. As a consequence we have from the above considerations that the $\mathbf{CLS}$ is exponentially stable, which constitutes the principal result of Section \ref{sec_4b}. Finally, the implications of our study for future research are discussed in the conclusions, Section \ref{sec_5}.
%

\section{Spectrum and asymptotics of eigenvalues}\label{sec_3}

The spectral problem for the operator pencil given by \eqref{eq1lipenc12} is
\begin{equation}\label{eq_10}
\mathcal{P}\left(\lambda\right)x=\left(\lambda I-T\right)x=0,\quad x\in \bm{D}\left({A}\right),\quad \lambda\in \mathbf{C}.
\end{equation}
We can conclude from Theorem \ref{L-2-1} that the $\mathbf{CLS}$ operator $T$ has a compact resolvent ($0\in\varrho\left(T\right)$ and $T^{-1}$ is compact) and so, as we have stated in the paragraph preceding Definition \ref{def01}, this implies that the spectrum is discrete,
\begin{equation*}
\sigma_0\left(T\right)=\sigma\left(T\right)=\sigma\left(\mathcal{M}\right)=\sigma_0\left(\mathcal{M}\right),
\end{equation*}
consisting solely of normal eigenvalues which accumulate only at infinity and which coincide with the eigenvalues from the boundary-eigenvalue problem \eqref{eq_1sffss1}. The following is a useful theorem on the location of the eigenvalues.
\begin{theorem}\label{T-3-1}
The spectrum of the $\mathbf{CLS}$ operator $T$ is symmetric with respect to the real axis and lies in the closed left half-plane, excluding the origin. When $\eta,\kappa>0$ the spectrum is confined to the open left half-plane.
\end{theorem}
\begin{proof}
Let ${x}\left( \neq 0\right)$ be an eigenvector of $T$ corresponding to an eigenvalue ${\lambda} $. Then, from \eqref{eq_10},
\begin{equation*}
\left(\lambda I-T\right)x =0.
\end{equation*}
With $T$ as given by \eqref{eq191922}--\eqref{eq_09xc} we have 
\begin{equation*}
\overline{\left(\lambda I-T\right)x }=(\overline{\lambda } I-T)\,\overline{x }=0,
\end{equation*}
which means that $\overline{x}$ is an eigenvector corresponding to the eigenvalue $\overline{\lambda }$. This proves that the spectrum of $T$ is symmetric with respect to the real axis.

We now take the inner product of $\left(\lambda I-T\right)x $ with $x $ to obtain
\begin{equation*}
\left<\left(\lambda I-T\right)x ,x \right>=\lambda \left\|x \right\|^2-\left<Tx ,x \right>=0.
\end{equation*}
The real part of this equation is
\begin{equation}\label{eq_15}
\operatorname{Re}\left(\lambda\right)\left\|x \right\|^2-\operatorname{Re}\left<Tx ,x \right>=0.
\end{equation}
By Theorem \ref{L-2-1}
\begin{equation*}
\operatorname{Re}\left(\lambda\right) =\frac{\operatorname{Re}\left<Tx ,x \right>}{\left\|x \right\|^2}\leq 0,
\end{equation*}
proving that the spectrum of $T$ lies in the closed left half-plane. That the origin does not belong to the spectrum is also an immediate consequence of Theorem \ref{L-2-1}.

For the proof of the second statement, namely that when $\eta,\kappa>0$ then $\operatorname{Re}\left(\lambda\right) <0$, suppose ${\lambda }$ is a purely imaginary eigenvalue, $\operatorname{Re}\left(\lambda\right) =0$, with eigenvector $x$. Then from \eqref{eq_15} we have $\operatorname{Re}\left<Tx ,x \right>=0$ and it follows from \eqref{eq_13} that
\begin{equation*}
-\beta\eta\left|v \left(1\right)\right|^2-\kappa\,|v '\left(1\right)\!|^2=0.
\end{equation*}
This implies, since $\beta\in\left(0,1\right)$ and $\eta,\kappa>0$ by assumption,
\begin{equation*}
|v '\left(1\right)\!|^2=0,\quad \left|v \left(1\right)\right|^2=0
\end{equation*}
and, as 
\begin{equation*}
x =\left(\begin{matrix}
w \\
v 
\end{matrix}\right)
\end{equation*}
is an eigenvector, so $v =\lambda w $,
\begin{equation*}
\left|\lambda \right|^2|w '\left(1\right)\!|^2=0,\quad \left|\lambda \right|^2\left|w \left(1\right)\right|^2=0.
\end{equation*}
Therefore $w \left(1\right)=w '\left(1\right)=0$ since $\left|\lambda \right|>0$ by Theorem \ref{L-2-1}. In this case, setting $\lambda=i\mu$, $\mu\in \mathbf{R}$, $w=w\left(s\right)$ satisfies the boundary-eigenvalue problem
\begin{equation*}
\begin{split}
w^{(4)}-(\gamma-\eta^2)\,w''+2i\beta\eta\mu w'-\mu^2w&=0,\\
w\left(0\right)=w''\left(0\right)&=0,\\
 w\left(1\right)=w'\left(1\right)=w''\left(1\right)=w^{(3)}\left(1\right)&=0.
\end{split}
\end{equation*}
It follows that only trivial solutions can result from the boundary conditions at $s=1$. However, this contradicts the assumption that
\begin{equation*}
{x} =\left(\begin{matrix}
w \\
v 
\end{matrix}\right)=\left(\begin{matrix}
w \\
\lambda w 
\end{matrix}\right)
\end{equation*}
is an eigenvector. Consequently, we must have $\operatorname{Re}\left(\lambda\right) <0$ as $\eta,\kappa>0$.
\end{proof}

As has been stated repeatedly in the paper, $\lambda$ is in the spectrum $\sigma\left(T\right)$ (and
hence is an eigenvalue) of $T$ if and only if, for that value of $\lambda$, the boundary-eigenvalue problem \eqref{eq_1sffss1} has a nontrivial solution. In the next subsection, we discuss eigenvalue asymptotics, by restricting attention to \eqref{eq_1sffss1}.

\subsection{Asymptotics of eigenvalues}\label{sec3}

Our goal in this subsection is to produce explicit asymptotic expansions of the eigenvalues with large modulus of the boundary-eigenvalue problem \eqref{eq_1sffss1} for $\eta,\kappa>0$, and thereby obtain suitable (for our purposes) asymptotic estimates of their location. This requires only standard techniques of complex analysis after a preliminary development of asymptotic expaniosn for the solutions of the differential equation in \eqref{eq_1sffss1} (see \cite[Section II.4]{Naimark1967} for more details on the underlying method).

According to Theorem \ref{T-3-1} all eigenvalues lie in the left half-plane and those with nonzero imaginary part must come in conjugate pairs. So we need only consider the boundary-eigenvalue problem in the second quadrant of the complex plane corresponding to eigenvalues with $\operatorname{Im}\left(\lambda\right)\ge 0$. The sector
\begin{equation*}
\mathscr{S}_0\coloneqq\left\{\rho\in\mathbf{C}~\middle|~0\leq \arg\left(\rho\right) \leq\frac{\pi}{4}\right\}
\end{equation*}
corresponds to this quadrant under the spectral parameter transformation $\lambda\mapsto i\rho^2$. For $\mathscr{S}$ the four roots of $-1$ have the ordering
\begin{equation*}
\operatorname{Re}\left(-\rho\right)\leq \operatorname{Re}\left(i\rho\right)\leq \operatorname{Re}\left(-i\rho\right) \leq \operatorname{Re}\left(\rho\right),\quad \rho\in\mathscr{S}.
\end{equation*}
Obviously, for $\rho\in\mathscr{S}$, the inequalities $\operatorname{Re}\left(-\rho\right)=-\left|\rho\right|\cos\left(\arg\left(\rho\right)\right)\leq-c\left|\rho\right|$ and $\operatorname{Re}\left(i\rho\right)=-\left|\rho\right|\sin\left(\arg\left(\rho\right)\right)\leq 0$ then hold with a constant $c>0$. Thus, asymptotically in $\mathscr{S}$, i.e.\ for large $\left|\rho\right|$, we have that $\left|e^{i\rho}\right|\leq1$ and $\left|e^{-\rho }\right|=\mathscr{O}\,(e^{-c\left|\rho\right|})$, which we shall use subsequently. First we prove a lemma.
\begin{lemma}\label{L-4-1}
In the sector $\mathscr{S}$, there exists a fundamental system $\left\{w_r\left(\lambda,\,\cdot\,\right)\right\}^4_{r=1}$ of the differential equation in \eqref{eq_1sffss1} which, under the spectral parameter transformation $\lambda\mapsto i\rho^2$ associated with $\rho\in\mathscr{S}$, has the following asymptotic expansion for large $\left|\rho\right|$:
\begin{equation*}
w^{(k)}_r\left(\rho,s\right)=\left(i^r\rho \right)^ke^{i^r\rho s}\left(1+{{f}_r\left(s\right)}+\frac{i^r{f}_{r1}\left(s\right)+k{f}'_{r}\left(s\right)}{i^r\rho}+\mathscr{O}\,(\rho^{-2})\right),\quad r=1,2,3,4,
\end{equation*}
for $k=0,1,2,3$, $s\in\left[0,1\right]$, where
\begin{equation*}
f_r\left(s\right)=-1+e^{-\left(-1\right)^{r}\frac{i\beta\eta}{2} s},\quad {f}_{r1}\left(s\right)= \frac{\left(-i\right)^r}{4}\left(\frac{\beta^2\eta^2}{2}+\gamma-\eta^2\right) s e^{-\left(-1\right)^{r}\frac{i\beta\eta}{2}s},\quad r=1,2,3,4.
\end{equation*}
\end{lemma}
\begin{proof}
If in the differential equation in \eqref{eq_1sffss1} we make the change from $\lambda$ to $i\rho^2$, then it may be rewritten as
\begin{equation}\label{eqsss_07}
w^{(4)}-(\gamma-\eta^2)\,w''+2i\rho^2\beta\eta w'-\rho^4w=0.
\end{equation}
Take $\left\{w_r\left({\rho},\,\cdot\,\right)\right\}^4_{r=1}$ to be a fundamental system of \eqref{eqsss_07} having the following expansion:
\begin{equation}\label{eqrel023ss}
w_r\left(\rho,s\right)=e^{\rho\omega_r s}\left(1+f_r\left(s\right)+\frac{f_{r1}\left(s\right)}{\rho}+\frac{H_r\left(\rho,s\right)}{\rho^2} \right),\quad r=1,2,3,4,
\end{equation}
for $s\in\left[0,1\right]$, where $\omega_r=i^{r}$ and $f_r$, $f_{r1}$ and $ H_r\left(\rho,\,\cdot\,\right)$ are smooth functions, say analytic on $\left[0,1\right]$. Substituting \eqref{eqrel023ss} in \eqref{eqsss_07} and collecting terms according to powers of $\rho$, we get
\begingroup
\allowdisplaybreaks
\begin{align*}
0&=\rho^3\,\Bigl[4\omega^3_r f_r' +2i\beta\eta \omega_r \left(1+f_r \right)\Bigr] \\
&\qquad +\rho^2\,\Bigl[4\omega^3_r f_{r1}' +2i\beta\eta \omega_r f_{r1}+ 6\omega^2_rf''_r +2i\beta\eta f'_r -(\gamma-\eta^2)\,\omega_r^2 \left(1+f_r \right)\Bigr]\\
&\qquad +\rho\,\Bigl[4\omega^3_r H'_r\left(\rho,\,\cdot\,\right)+2i\beta\eta \omega_r H_r\left(\rho,\,\cdot\,\right)+6\omega^2_rf''_{r1} +2i\beta\eta f'_{r1} -(\gamma-\eta^2)\,\omega_r^2 f_{r1} \\
&\qquad+4\omega_r f^{(3)}_r -2\,(\gamma-\eta^2)\,\omega_r f'_r \Bigr]+6\omega^2_rH''_r\left(\rho,\,\cdot\,\right) +2i\beta\eta H'_r\left(\rho,\,\cdot\,\right)\\
&\qquad-(\gamma-\eta^2)\,\omega_r^2 H_r\left(\rho,\,\cdot\,\right)+ 4\omega_r f^{(3)}_{r1} -2\,(\gamma-\eta^2) \,\omega_r f'_{r1} +f^{(4)}_r -(\gamma-\eta^2)\,f''_r \\
&\qquad+\rho^{-1}\,\Bigl[f^{(4)}_{r1} -(\gamma-\eta^2)\, f''_{r1} +4\omega_r H^{(3)}_r\left(\rho,\,\cdot\,\right)-2\,(\gamma-\eta^2)\,\omega_r H'_r\left(\rho,\,\cdot\,\right)\Bigr]\\
&\qquad+\rho^{-2}\,\Bigl[H^{(4)}_r\left(\rho,\,\cdot\,\right)-(\gamma-\eta^2)\,H''_r\left(\rho,\,\cdot\,\right)\Bigr].
\end{align*}
\endgroup
Let us consider from this the following system of differential equations for $f_r$, $f_{r1}$:
\begin{equation}\label{diffsys01}
\left\{\begin{split}
4\omega^3_r f_r' +2i\beta\eta \omega_r \left(1+f_r \right)&=0,\\
4\omega^3_r f_{r1}' +2i\beta\eta \omega_r f_{r1} + 6\omega^2_rf''_r +2i\beta\eta f'_r -(\gamma-\eta^2)\,\omega_r^2 \left(1+f_r \right)&=0.
\end{split}\right.
\end{equation}
Solving this system \eqref{diffsys01} with $f_r\left(0\right)=0$ and $f_{r1}\left(0\right)=0$ we have
\begin{equation}\label{eqx32}
f_r\left(s\right)=-1+e^{-\frac{i\beta\eta}{2}\omega^2_r s},\quad
 f_{r1}\left(s\right)= \frac{\omega^3_r}{4}\left(\frac{\beta^2\eta^2}{2}+\gamma-\eta^2\right) s e^{-\frac{i\beta\eta}{2}\omega^2_r s}
\end{equation}
for $s\in\left[0,1\right]$. So if $f_{r1}\left(s\right)$ satisfies the second differential equation in \eqref{diffsys01} then, with the first rewritten as $f_r'=-\frac{i\beta\eta}{2} \omega_r^2 \left(1+f_r \right)$, it is seen to satisfy the differential equation
\begin{equation}\label{eqxdsdd32}
 f_{r1}' +\frac{i\beta\eta}{2} \omega^2_r f_{r1}+\frac{\omega_r}{2i\beta\eta}\left(\frac{\beta^2\eta^2}{2}+\gamma-\eta^2\right)f'_r=0
\end{equation}
as well. The substitution
\begin{equation*}
f'_{r}\left(s\right)=-\frac{i\beta\eta}{2}\omega^2_r e^{-\frac{i\beta\eta}{2}\omega_r^2s}
\end{equation*}
from \eqref{eqx32}, along with the fact that
\begin{equation*}
\left(f'_{r1}\left(s\right)+\frac{i\beta\eta}{2}\omega_r^2f_{r1}\left(s\right)\right)e^{\frac{i\beta\eta}{2}\omega_r^2s}=\frac{d}{ds}\left(f_{r1}\left(s\right)e^{\frac{i\beta\eta}{2}\omega_r^2s}\right),
\end{equation*}
transforms \eqref{eqxdsdd32} into
\begin{equation*}
\left(f_{r1}e^{\frac{i\beta\eta }{2}\omega_r^2s}\right)'-\frac{\omega^3_r}{4}\left(\frac{\beta^2\eta^2}{2} + \gamma-\eta^2\right)= 0.
\end{equation*}
From this we get that the $ H_r\left(\rho,s\right)$ are uniformly bounded, say $\left|H_r\left(\rho,s\right)\right|\leq N$, $r=1,2,3,4$, some constant $N>0$, with respect to $s\in [0,1]$ for $\rho\in \mathscr{S}$ as $\left|\rho\right|$ becomes large. The rest of the proof follows the usual lines. On differentiating \eqref{eqrel023ss} $k$ times, factoring out the terms $\left(\rho \omega_{r}\right)^k$, and collecting terms of $\mathscr{O}\,(\rho^{-2})$, we obtain
\begin{equation*}
w^{(k)}_r\left(\rho,s\right)=\left(\rho\omega_r \right)^ke^{\rho\omega_r s}\left(1+{{f}_r\left(s\right)}+\frac{\omega_r{f}_{r1}\left(s\right)+k{f}'_{r}\left(s\right)}{\rho\omega_r}+\mathscr{O}\,(\rho^{-2})\right),\quad r=1,2,3,4,
\end{equation*}
for $k=0,1,2,3$. Recalling $\omega_r=i^{r}$ and noting $\omega^2_r=\left(-1\right)^r$, $\omega^3_r=\left(-i\right)^r$ in \eqref{eqx32} gives the desired result.
\end{proof}

\begin{theorem}\label{T-4-1}
The nonreal eigenvalues of the boundary-eigenvalue problem \eqref{eq_1sffss1} are geometrically and algebraically simple and asymptotically, i.e.\ for large $n\in\mathbf{N}$,
\begin{equation}\label{eqasymff567}
\lambda_n=i{\rho}_n^2,\quad {\rho}_n=\left(n-\frac{1}{2}\right)\pi
+\left[\frac{\frac{\beta^2\eta^2}{2}+\gamma-\eta^2}{2}+i\left(\beta\eta+ \frac{1}{\kappa}\right)\right]\frac{1}{2\left(n-\frac{1}{2}\right)\pi}+\mathscr{O}\,(n^{-2}).
\end{equation}
\end{theorem}
\begin{proof}
Let $\lambda$ be an eigenvalue of \eqref{eq_1sffss1} and $w\left(\lambda,\,\cdot\,\right)$ be a corresponding eigenfunction. Again replacing $\lambda$ by $i\rho^2$ for $\rho\in\mathscr{S}$, we have by Lemma \ref{L-4-1} the linear combination $w\left(\rho,\,\cdot\,\right)=\sum_{r=1}^4a_rw_r\left(\rho,\,\cdot\,\right)$ with respect to the fundamental system $\left\{w_r\left({\rho},\,\cdot\,\right)\right\}^4_{r=1}$, with coefficients $a_r$ possibly dependent on the spectral parameter. Insertion into the boundary conditions yields that $\lambda$ is an eigenvalue of \eqref{eq_1sffss1} if and only if for $\rho\in\mathscr{S}$ we have $\det M\left(\rho\right)=0$, where
\begingroup
\allowdisplaybreaks
\begin{align*}
M\left(\rho\right)&=\left(\begin{matrix}
w_1\left(\rho,0\right)&w_2\left(\rho,0\right)\\
w''_1\left(\rho,0\right)&w''_2\left(\rho,0\right)\\
w''_1\left(\rho,1\right)+i\kappa\rho^2 w'_1\left(\rho,1\right)&w''_2\left(\rho,1\right)+i\kappa\rho^2 w'_2\left(\rho,1\right)\\
w^{(3)}_1\left(\rho,1\right)-(\gamma-\eta^2)\, w'_1\left(\rho,1\right)&w^{(3)}_2\left(\rho,1\right)-(\gamma-\eta^2)\, w'_2\left(\rho,1\right)
\end{matrix}\right.\\[0.2em]
&\quad\left.\hspace{0.2\linewidth}
\begin{matrix}
w_3\left(\rho,0\right)&w_4\left(\rho,0\right)\\
w''_3\left(\rho,0\right)&w''_4\left(\rho,0\right)\\
w''_3\left(\rho,1\right)+i\kappa\rho^2 w'_3\left(\rho,1\right)&w''_4\left(\rho,1\right)+i\kappa\rho^2 w'_4\left(\rho,1\right)\\
w^{(3)}_3\left(\rho,1\right)-(\gamma-\eta^2)\, w'_3\left(\rho,1\right)&w^{(3)}_4\left(\rho,1\right)-(\gamma-\eta^2)\, w'_4\left(\rho,1\right)
\end{matrix}\right).
\end{align*}
\endgroup
Geometric simplicity of the eigenvalues follows from the fact that, as can be checked, the rank of $M\left(\rho\right)$ is 3, meaning that there is exactly one linearly independent solution of \eqref{eq_1sffss1}.

Now, suppose $\rho_0\in\mathscr{S}$ is a zero of $\det M\left(\,\cdot\,\right)$ and let $w\left(\rho_0,\,\cdot\,\right)$ be a corresponding eigenfunction of \eqref{eq_1sffss1} under the change from $\lambda$ to $i\rho^2$. Then, with $\hat{a}_r=\hat{a}_r\left(\rho_0\right)$, $r=1,2,3,4$, we write $w\left(\rho_0,\,\cdot\,\right)=\sum_{r=1}^4\hat{a}_r\left(\rho_0\right)w_r\left(\rho_0,\,\cdot\,\right)$, where we recall that $w\left(\rho,\,\cdot\,\right)=\sum_{r=1}^4a_rw_r\left(\rho,\,\cdot\,\right)$ with respect to the fundamental system $\left\{w_r\left({\rho},\,\cdot\,\right)\right\}^4_{r=1}$, the coefficients ${a}_r={a}_r\left(\rho\right)$ being analytic for $\rho$ in a neighbourhood of $\rho_0$, and ${a}_r\left(\rho_0\right)=\hat{a}_r\left(\rho_0\right)$. Then in a neighbourhood of $\rho_0$,
\begin{align*}
&{a}_1\left(\rho\right)\det M\left(\rho\right)\\
&\quad=\left|\begin{matrix}
w\left(\rho,0\right)&w_2\left(\rho,0\right)\\
w''\left(\rho,0\right)&w''_2\left(\rho,0\right)\\
w''\left(\rho,1\right)+i\kappa\rho^2 w'\left(\rho,1\right)&w''_2\left(\rho,1\right)+i\kappa\rho^2 w'_2\left(\rho,1\right)\\
w^{(3)}\left(\rho,1\right)-(\gamma-\eta^2)\, w'\left(\rho,1\right)&w^{(3)}_2\left(\rho,1\right)-(\gamma-\eta^2)\, w'_2\left(\rho,1\right)
\end{matrix}\right.\\[0.2em]
&\quad\left.\hspace{0.2\linewidth}
\begin{matrix}
w_3\left(\rho,0\right)&w_4\left(\rho,0\right)\\
w''_3\left(\rho,0\right)&w''_4\left(\rho,0\right)\\
w''_3\left(\rho,1\right)+i\kappa\rho^2 w'_3\left(\rho,1\right)&w''_4\left(\rho,1\right)+i\kappa\rho^2 w'_4\left(\rho,1\right)\\
w^{(3)}_3\left(\rho,1\right)-(\gamma-\eta^2)\, w'_3\left(\rho,1\right)&w^{(3)}_4\left(\rho,1\right)-(\gamma-\eta^2)\, w'_4\left(\rho,1\right)
\end{matrix}\right|.
\end{align*}
Set ${u}\left(\rho,s\right)= (\partial {w}/\partial \rho)\left(\rho,s\right)$, assuming that ${u}\left(\rho, \,\cdot\,\right)$ is an associated function. We calculate
\begingroup
\allowdisplaybreaks
\begin{align*}
&{a}_1\left(\rho_0\right)\left.\frac{d}{d\rho}\left(\det M\left(\rho\right)\right)\right|_{\rho=\rho_0}\\
&\quad=\left|\begin{matrix}
{u}\left(\rho_0,0\right)&w_2\left(\rho_0,0\right)\\
{u}''\left(\rho_0,0\right)&w''_2\left(\rho_0,0\right)\\
{u}''\left(\rho_0,1\right)+i\kappa\rho_0^2 {u}'\left(\rho_0,1\right)+2i\kappa\rho_0w'\left(\rho_0,1\right)&w''_2\left(\rho_0,1\right)+i\kappa\rho_0^2 w'_2\left(\rho_0,1\right)\\
{u}^{(3)}\left(\rho_0,1\right)-(\gamma-\eta^2)\, {u}'\left(\rho_0,1\right)&w^{(3)}_2\left(\rho_0,1\right)-(\gamma-\eta^2)\, w'_2\left(\rho_0,1\right)
\end{matrix}\right.\\[0.2em]
&\quad\left.\hspace{0.2\linewidth}
\begin{matrix}
w_3\left(\rho_0,0\right)&w_4\left(\rho_0,0\right)\\
w''_3\left(\rho_0,0\right)&w''_4\left(\rho_0,0\right)\\
w''_3\left(\rho_0,1\right)+i\kappa\rho_0^2 w'_3\left(\rho_0,1\right)&w''_4\left(\rho_0,1\right)+i\kappa\rho_0^2 w'_4\left(\rho_0,1\right)\\
w^{(3)}_3\left(\rho_0,1\right)-(\gamma-\eta^2)\, w'_3\left(\rho_0,1\right)&w^{(3)}_4\left(\rho_0,1\right)-(\gamma-\eta^2)\, w'_4\left(\rho_0,1\right)
\end{matrix}\right|.
\end{align*}
\endgroup
It can be shown that the above determinant then does not vanish. Hence
\begin{equation*}
\det M\left(\rho_0\right)=0,\quad \left.\frac{d}{d\rho}\left(\det M\left(\rho\right)\right)\right|_{\rho=\rho_0}\ne 0
\end{equation*}
and we conclude that $\rho_0$ is a simple zero of $\det M\left(\,\cdot\,\right)$ (see \cite[Section I.2.1]{Naimark1967}). Since $\rho_0$ is an arbitrary zero of $\det M\left(\,\cdot\,\right)$, we have as a consequence that all of the nonreal zeros and hence the nonreal eigenvalues of \eqref{eq_1sffss1} are simple.

As for the eigenvalue asymptotics, we begin by putting $\Delta\left(\lambda\right)\coloneqq \det M\left(\rho\right)$ , freely mixing here and in the following the spectral parameters $\lambda$ and $\rho$, and calculating the asymptotic expansion of $\Delta\left(\lambda\right)$. Since $\frac{\beta^2\eta^2}{2}+{\gamma}-{\eta}^2$ appears repeatedly in the calculations to follow, we economise the notation by putting
\begin{equation*}
\theta\coloneqq \frac{\beta^2\eta^2}{2}+{\gamma}-{\eta}^2.
\end{equation*}
Applying Lemma \ref{L-4-1} and its proof to the columns of $M\left(\rho\right)$, we calculate
\begingroup
\allowdisplaybreaks
\begin{align*}\label{eqwetr555}
w_r\left(\rho,0\right)&=1, \\
w''_r\left(\rho,0\right)&=\left(\rho\omega_r\right)^2\left[1-\frac{i\beta\eta}{\rho}\omega_r\right]_2,\\
w''_r\left(\rho,1\right)+i\kappa\rho^2 w'_r\left(\rho,1\right)&=i\kappa\rho^2\left(\rho\omega_r\right)e^{\rho\omega_r}e^{-\frac{i\beta\eta}{2}\omega^2_r}\left[1+\frac{1}{4\rho\omega_r}\,\left(\theta-2i\beta\eta\omega_r^2\right)+\frac{\omega_r}{i\kappa\rho}\right]_2,\\
w^{(3)}_r\left(\rho,1\right)-(\gamma-\eta^2)\,w'_r\left(\rho,1\right)&=\left(\rho\omega_r\right)^3e^{\rho\omega_r}e^{-\frac{i\beta\eta}{2}\omega^2_r}\left[1+\frac{1}{4\rho\omega_r}\,\left( \theta-6i\beta\eta\omega_r^2\right)\right]_2
\end{align*}
\endgroup
for $r=1,2,3,4$, where we have used the obvious standard notation $\left[a\right]_2\coloneqq a+\mathscr{O}\,(\rho^{-2})$. Substituting these expressions in $\Delta\left(\lambda\right)$ and dividing out the common factors of the rows and the columns, recalling that in $\mathscr{S}$, since $\omega_r=i^r$ and $\omega^2_r=\left(-1\right)^r$, we have that $\omega_1=-\omega_3=i$, $\omega_2=-\omega_4=-1$ and $\omega_1^2=\omega_3^2=-1$, $\omega_2^2=\omega_4^2=1$, we find
\begin{equation*}
\Delta\left(\lambda\right)=i\kappa\rho^8e^{\rho\omega_4}\det\left(
\Delta_1\left(\rho\right) , \Delta_2\left(\rho\right) , \Delta_3\left(\rho\right) , \Delta_4\left(\rho\right)
\right)
\end{equation*}
where the column vectors $\Delta_r\left(\rho\right)$ are defined as follows:
\begingroup
\allowdisplaybreaks
\begin{align*}
\Delta_1\left(\rho\right)&=\left(\begin{matrix}
e^{-\rho\omega_1}e^{\frac{i\beta\eta}{2}\omega_1^2}\\
\omega_1^2e^{-\rho\omega_1}e^{\frac{i\beta\eta}{2}\omega_1^2}\left[1-\frac{i\beta\eta}{\rho}\omega_1\right]_2 \\
\omega_1\left[1+\frac{1}{4\rho\omega_1}\left( \theta -2i\beta\eta\omega_1^2\right)+\frac{\omega_1}{i\kappa\rho}\right]_2\\
 \omega_1^3\left[1+\frac{1}{4\rho\omega_1}\left( \theta-6i\beta\eta\omega_1^2\right)\right]_2
\end{matrix}\right),\\
\Delta_2\left(\rho\right)&=\left(\begin{matrix}
e^{\frac{i\beta\eta}{2}\omega_2^2} \\
 \omega_2^2e^{\frac{i\beta\eta}{2}\omega_2^2}\left[1-\frac{i\beta\eta}{\rho}\omega_2\right]_2 \\
 \omega_2e^{\rho\omega_2}\left[1+\frac{1}{4\rho\omega_2}\left( \theta -2i\beta\eta\omega_2^2\right)+\frac{\omega_2}{i\kappa\rho}\right]_2 \\
 \omega_2^3e^{\rho\omega_2}\left[1+\frac{1}{4\rho\omega_2}\left( \theta-6i\beta\eta\omega_2^2\right)\right]_2
\end{matrix}\right),\\
\Delta_3\left(\rho\right)&=\left(\begin{matrix}
 e^{-\rho\omega_3}e^{\frac{i\beta\eta}{2}\omega_3^2}\\
 \omega_3^2e^{-\rho\omega_3}e^{\frac{i\beta\eta}{2}\omega_3^2}\left[1-\frac{i\beta\eta}{\rho}\omega_3\right]_2 \\
 \omega_3\left[1+\frac{1}{4\rho\omega_3}\left( \theta -2i\beta\eta\omega_3^2\right)+\frac{\omega_3}{i\kappa\rho}\right]_2\\
 \omega_3^3\left[1+\frac{1}{4\rho\omega_3}\left( \theta-6i\beta\eta\omega_3^2\right)\right]_2
\end{matrix}\right),\\
\Delta_4\left(\rho\right)&=\left(\begin{matrix}
 e^{-\rho\omega_4}e^{\frac{i\beta\eta}{2}\omega_4^2}\\
 \omega_4^2e^{-\rho\omega_4}e^{\frac{i\beta\eta}{2}\omega_4^2}\left[1-\frac{i\beta\eta}{\rho}\omega_4\right]_2 \\
 \omega_4\left[1+\frac{1}{4\rho\omega_4}\left( \theta -2i\beta\eta\omega_4^2\right)+\frac{\omega_4}{i\kappa\rho}\right]_2\\
 \omega_4^3\left[1+\frac{1}{4\rho\omega_4}\left( \theta-6i\beta\eta\omega_4^2\right)\right]_2
\end{matrix}\right).
\end{align*}
\endgroup
Hence,
\begin{equation*}
\Delta_r\left(\rho\right)=\Delta_{r,0}\left(\rho\right)+\mathscr{O}\,(e^{-c\left|\rho\right|}),\quad r=1,2,3,4,
\end{equation*}
where
\begingroup
\allowdisplaybreaks
\begin{align*}
\Delta_{1,0}\left(\rho\right)&=\left(\begin{matrix}
e^{-i\rho}e^{-\frac{i\beta\eta}{2}}\\
-e^{-i\rho}e^{-\frac{i\beta\eta}{2}}\left[1+\frac{\beta\eta}{\rho}\right]_2 \\
\left[i+\frac{1}{4\rho}\left( \theta +2i\beta\eta\right)+\frac{i}{\kappa\rho}\right]_2\\
\left[-i-\frac{1}{4\rho}\left( \theta+6i\beta\eta\right)\right]_2
\end{matrix}\right),\\
\Delta_{2,0}\left(\rho\right)&=\left(\begin{matrix}
e^{\frac{i\beta\eta}{2}} \\
 e^{\frac{i\beta\eta}{2}}\left[1+\frac{i\beta\eta}{\rho}\right]_2 \\
0\\
0 
\end{matrix}\right),\\
\Delta_{3,0}\left(\rho\right)&=\left(\begin{matrix}
 e^{i\rho}e^{-\frac{i\beta\eta}{2}}\\
 -e^{i\rho}e^{-\frac{i\beta\eta}{2}}\left[1-\frac{\beta\eta}{\rho}\right]_2 \\
 \left[-i+\frac{1}{4\rho}\left( \theta +2i\beta\eta\right)+\frac{i}{\kappa\rho}\right]_2\\
 \left[i-\frac{1}{4\rho}\left( \theta+6i\beta\eta\right)\right]_2
\end{matrix}\right),\\
\Delta_{4,0}\left(\rho\right)&=\left(\begin{matrix}
 0\\
 0 \\
 \left[1+\frac{1}{4\rho}\left( \theta -2i\beta\eta\right)+\frac{1}{i\kappa\rho}\right]_2\\
 \left[1+\frac{1}{4\rho}\left( \theta-6i\beta\eta\right)\right]_2
\end{matrix}\right).
\end{align*}
\endgroup
Recall that $e^{\rho\omega_2}=e^{-\rho\omega_4}=e^{-\rho}$ and $e^{\rho\omega_1}=e^{-\rho\omega_3}=e^{i\rho}$ for $\rho\in\mathscr{S}$. We find on reducing the determinant, after some calculations, that the characteristic equation $h\left(\rho\right)=\frac{1}{8}\kappa^{-1}\rho^{-8}e^{-\rho}\Delta\left(\lambda\right)=0$ has, thus, an asymptotic expansion of the form
\begin{equation}\label{eq223dc}
\cos\rho+\frac{\theta\left(\sin\rho+\cos\rho\right)}{4\rho}+\frac{i\left(\beta\eta+\frac{1}{\kappa}\right)\left(\sin\rho-\cos\rho\right)}{2\rho}+\mathscr{O}\,(\rho^{-2})=0
\end{equation}
or
\begin{equation*}
\cos\rho+\mathscr{O}\,(\rho^{-1})=0,
\end{equation*}
valid for $\rho$ in small neighbourhoods of $\left(n-\frac{1}{2}\right)\pi$ for large $n\in\mathbf{N}$. Set $f\left(\rho\right)= \cos\rho$. It follows that $g\left(\rho\right)= f\left(\rho\right)-h\left(\rho\right) =\mathscr{O}\,(\rho^{-1})$. Around $\tau_n\coloneqq\left(n-\frac{1}{2}\right)\pi$, large $n\in\mathbf{N}$, on a contour
\begin{equation*}
\Gamma\coloneqq\left\{\rho\in\mathbf{C}~\middle|~\left|\rho-\tau_n\right|=\frac{\pi}{4}\right\},
\end{equation*}
we consider $f$, so that on $\Gamma$, $\left|f\right|\geq \frac{1}{2}$ and the estimate $\left|g\right|< \left|f\right|$ holds. The functions $f$, $g$, $h$ are all analytic on $\Gamma$ and its interior. Inside $\Gamma$, by Rouche's theorem, there is exactly one zero of $f$ (hence exactly one zero of $h$), given by
\begin{equation}\label{eqwsyxc4}
{\rho}_n=\left(n-\frac{1}{2}\right)\pi+\xi_n.
\end{equation}
We can express the $\xi_n$ in the following way. First we substitute ${\rho}_n$ for $\rho$ in \eqref{eq223dc} and note that
\begin{align*}
\cos{\rho}_n&=\cos\left(n-\frac{1}{2}\right)\pi\cos \xi_n -\sin\left(n-\frac{1}{2}\right)\pi\sin \xi_n = -\left(-1\right)^{n-1}\sin \xi_n ,\\
\sin{\rho}_n&=\sin\left(n-\frac{1}{2}\right)\pi\cos \xi_n +\cos\left(n-\frac{1}{2}\right)\pi\sin \xi_n = \left(-1\right)^{n-1}\cos \xi_n .
\end{align*}
Then we obtain that the $\xi_n$ satisfy
\begin{equation*}
\sin \xi_n =\left[\frac{\theta}{2}+i\left(\beta\eta+ \frac{1}{\kappa}\right)\right]\frac{\cos \xi_n }{2{\rho}_n}
+\mathscr{O}\,({\rho}_n^{-2}).
\end{equation*}
Thus $\sin \xi_n\sim \xi_n$, $\cos \xi_n\sim 1$ and we have
\begin{equation}\label{eqwsyxc3}
\xi_n=\left[\frac{\theta}{2}+i\left(\beta\eta+ \frac{1}{\kappa}\right)\right]\frac{1}{2\left(n-\frac{1}{2}\right)\pi}+\mathscr{O}\,(n^{-2}).
\end{equation}
This together with \eqref{eqwsyxc4} and recalling the definition of $\theta$ (replace $\theta$ by $\frac{\beta^2\eta^2}{2}+{\gamma}-{\eta}^2$ in \eqref{eqwsyxc3}) gives  \eqref{eqasymff567}, thus completing the proof of the theorem.
\end{proof}

Theorem \ref{T-4-1} implies the following conclusion.
\begin{corollary}\label{rem01}
The following statements hold:
\begin{enumerate}[\normalfont(1)]
\item $\lambda=i\rho^2\in\mathbf{C}$ is an eigenvalue of the boundary-eigenvalue problem \eqref{eq_1sffss1} if and only if $\Delta\left(\lambda\right)=0$. Hence,
\begin{equation*}
\sigma\left(T\right)=\left\{\lambda=i\rho^2\in\mathbf{C}~\middle|~\Delta\left(\lambda\right)=0\right\}.
\end{equation*}
\item There is a finite integer $n_0$ such that
\begin{equation*}
\sigma\left(T\right)=\left\{\lambda_{n},\overline{\lambda_n}~\middle|~n\in\mathbf{N},~n\ge n_0\right\}
\end{equation*}
where the pairs $\lambda_n$ satisfy \eqref{eqasymff567}. (By standard reasoning using Rouche's theorem the number $n_0$ can be determined.)
\item The $\lambda_{n}$, $\overline{\lambda_n}$ are simple zeros of $\Delta$ and thus simple eigenvalues of $T$, i.e., $\nu\left(\lambda_n\right)=\operatorname{dim}\operatorname{ker}\left(\lambda_n I-T\right)=1$.
\item Using \eqref{eqasymff567},
\begin{equation*}
\lambda_n=-\left(\beta\eta+\frac{1}{\kappa}\right)+i\left\{\left[\left(n-\frac{1}{2}\right)\pi\right]^2+\frac{ \frac{\beta^2\eta^2}{2}+\gamma-\eta^2 }{2}\right\}+\mathscr{O}\,(n^{-1}).
\end{equation*}
\item $\operatorname{Re}\left(\lambda_n\right)=-\left(\beta\eta+\frac{1}{\kappa}\right)$ is the vertical asymptote of $\sigma\left(T\right)$.
\item For large $n,m\in\mathbf{N}$, the eigenvalues of $T$ have a gap, i.e.,
\begin{equation*}
\inf_{n\neq m}\left|\lambda_n-\lambda_m\right|>0.
\end{equation*}
\end{enumerate}
\end{corollary}

So, to summarise, we have effectively shown so far that the boundary-eigenvalue problem \eqref{eq_1sffss1} has only normal eigenvalues, which form a countably infinite set and accumulate only at infinity, with nonreal eigenvalues properly enumerated as ${\lambda}_{\pm n}$, where $\lambda_{-n}=\overline{\lambda_n}$ and ordered $ \operatorname{Im}\left(\lambda_n\right)\leq\operatorname{Im}\left(\lambda_{n+1}\right)$. The enumeration is such that in the sequence $\left\{{\lambda}_{\pm n}\right\}_{n=1}^\infty$ we have
$\sup\left|\operatorname{Re}\left(\lambda_{\pm n}\right)\right|<\infty$. In fact this combined with the gap condition in the corollary shows that the sequence is interpolating. (It should be emphasised that, in general, the proper enumeration of the eigenvalues is crucial for basisness results, but as a set it does not matter of course.)

\section{Riesz basisness}\label{sec_4}

In investigating the Riesz basisness of the root functions and root vectors corresponding to the eigenvalues of the boundary-eigenvalue problem \eqref{eq_1sffss1}, we now know that all nonreal eigenvalues are simple so that the associated eigenprojections have rank 1 and, as a consequence, that there are only corresponding eigenfunctions and eigenvectors and no associated ones. Indeed, since, by Theorem \ref{T-4-1}, the algebraic multiplicity of each nonreal eigenvalue coincides with the geometric multiplicity and is 1, each system of chains of root vectors consists of exactly one chain of root vectors, with its length being the algebraic multiplicity of the corresponding eigenvalue, namely 1. It is possible that there exist purely real eigenvalues. However, we note that if it is assumed that there are no purely real eigenvalues, then this assumption is without loss of generality for the purposes of the section. According to Corollary \ref{rem01} and the paragraph following it, we can indeed say that if there are purely real eigenvalues, at most finitely many of which are not simple. Then the nonsimplicity of the eigenvalues would not affect the Riesz basisness results. (See also Remark \ref{rmrkfin01}.)

Two principal results are proven in this section. We show that the eigenfunctions of the boundary-eigenvalue problem \eqref{eq_1sffss1} form a Riesz basis for $\bm{L}_2\left(0,1\right)$. The second of these two results is to show that the eigenvectors of $T$ form a Riesz basis for the space $\mathbb{X}$. The main tool in our analysis will be Bari's theorem mentioned in the last paragraph of Section \ref{sec2_3}: The proofs are based on the notion of the quadratic closeness (or ``quadratic nearness'' in Singer's terminology \cite[Definition II.11.2]{MR0298399}) of bases of eigenfunctions or eigenvectors to orthonormal ones, and the results then follow via an application of the theorem.
\begin{theorem}\label{T-4-2}
Let $\left\{\lambda_n,\overline{\lambda_n}\right\}_{n=1}^\infty$ be the zeros of $\Delta$ which satisfy \eqref{eqasymff567}. Then there is a corresponding sequence of eigenfunctions $\{w\left(\rho_n, \,\cdot\,\right),\overline{w\left(\rho_n, \,\cdot\,\right)}{\}}_{n=1}^\infty$ of \eqref{eq_1sffss1}, $\left\|w\left(\rho_n,\,\cdot\,\right)\right\|_0=1$, such that the sequence $\{w\left(\rho_n, \,\cdot\,\right){\}}_{n=1}^\infty$ forms a Riesz basis for $\bm{L}_2\left(0,1\right)$.
\end{theorem}
\begin{proof}
The proof will proceed in three steps.
\begin{enumerate}[label=,leftmargin=*,align=left,labelwidth=\parindent,labelsep=0pt,ref={\arabic*}]
\item\label{item06}\textbf{Step 1.} We begin by considering the boundary-eigenvalue problem
\begin{equation}\label{eq1ssd3cc2}
\left\{\begin{split}
\doublehat{w}^{(4)}-(\gamma-\eta^2)\, \doublehat{w}''+{\lambda}^2\doublehat{w}&=0,\\
\doublehat{w}\left(0\right)=\doublehat{w}''\left(0\right)&=0,\\
\doublehat{w}''\left(1\right)&=0,\\
\doublehat{w}^{(3)}\left(1\right)-(\gamma-\eta^2)\, \doublehat{w}'\left(1\right)&=0,
\end{split}\right.
\end{equation}
where we have $\lambda$-independent boundary conditions. Let $\doublehat{\lambda}_n=i\tilde{\tau}_n^2$ be an eigenvalue of \eqref{eq1ssd3cc2} with corresponding eigenfunction $\doublehat{w}\left(\tilde{\tau}_n , \,\cdot\,\right)$. It should be observed that \eqref{eq1ssd3cc2} is a selfadjoint problem in $\bm{L}_2$ with respect to the spectral parameter $\mu=-{\lambda}^2$ and admits only positive eigenvalues
\begin{equation*}
\mu_n=-\doublehat{\lambda}_n^2= \tilde{\tau}_n^4>0,\quad n\in\mathbf{N},
\end{equation*}
which can be ordered $ \mu_n\leq\mu_{n+1}$ with $\mu_n\rightarrow\infty$ as $n\rightarrow\infty$. Then, using a standard spectral theory result (see \cite[Theorem II.5.1]{Naimark1967}) we have that there exists a sequence $\{\doublehat{w}\left(\mu_n, \,\cdot\,\right){\}}_{n=1}^\infty$ of eigenfunctions of \eqref{eq1ssd3cc2} which forms an orthonormal basis for $\bm{L}_2\left(0,1\right)$. Using simple modifications of the proof of Theorem \ref{T-4-1}, for large $n\in\mathbf{N}$,
\begin{equation*}
\tilde{\tau}_n=\left(n-\frac{1}{2}\right)\pi
+\frac{\gamma-\eta^2}{4\left(n-\frac{1}{2}\right)\pi}+\mathscr{O}\,(n^{-2}).
\end{equation*}
Equivalently
\begin{equation*}
\tilde{\tau}_n={\tau}_n
+\frac{\gamma-\eta^2}{4\tau_n}+\mathscr{O}\,({\tau}_n^{-2}),\quad \tau_n\coloneqq \left(n-\frac{1}{2}\right)\pi.
\end{equation*}
A corresponding eigenfunction is given by
\begin{equation*}
w\left(\tilde{\tau}_n,s\right)=\left(\alpha^+_n\right)^2\sinh\alpha^+_n\sin\alpha^-_n s+\left(\alpha^-_n\right)^2\sin\alpha^-_n\sinh\alpha^+_n  s
\end{equation*}
where
\begin{equation*}
\alpha^\pm_n=\sqrt{\frac{\sqrt{4\tilde{\tau}_n+\gamma-\eta^2}\pm(\gamma-\eta^2)}{2}},
\end{equation*}
and we may set $\doublehat{w}\left(\mu_n,s\right)=\frac{w\left(\tilde{\tau}_n,s\right)}{\left\|w\left(\tilde{\tau}_n,\,\cdot\,\right)\right\|_0}$.
\item\label{item06s}\textbf{Step 2.} Consider now the boundary-eigenvalue problem
\begin{equation}\label{eq1ssd32}
\left\{\begin{split}
\hat{w}^{(4)}-(\gamma-\eta^2)\, \hat{w}''+{\lambda}^2\hat{w}&=0,\\
\hat{w}\left(0\right)=\hat{w}''\left(0\right)&=0,\\
\hat{w}''\left(1\right)+{\lambda}\kappa \hat{w}'\left(1\right)&=0,\\
\hat{w}^{(3)}\left(1\right)-(\gamma-\eta^2)\, \hat{w}'\left(1\right)&=0.
\end{split}\right.
\end{equation}
Let $\hat{\lambda}_n=i\tilde{\rho}_n^2$ be an eigenvalue of \eqref{eq1ssd32} with eigenfunction $\hat{w}\left(\tilde{\rho}_n , \,\cdot\,\right)$. It can be shown via a direct calculation that
\begin{equation*}
\tilde{\rho}_n=\tilde{\tau}_n
+\frac{i }{2\kappa\tilde{\tau}_n}+\mathscr{O}\,(\tilde{\tau}_n^{-2}),
\end{equation*}
and that
\begin{equation*}
\hat{w}\left(\tilde{\rho}_n ,s\right)=\doublehat{w}\left(\tilde{\tau}_n,s\right)\sum^{\infty}_{k=0}\frac{U_{k}\left(s\right)}{\tilde{\tau}_n^k}
\end{equation*}
where we recall that $\doublehat{w}\left(\tilde{\tau}_n, \,\cdot\,\right)$ is an eigenfunction of \eqref{eq1ssd3cc2} corresponding to $\tilde{\tau}_n$, $\left\|\doublehat{w}\left(\tilde{\tau}_n, \,\cdot\,\right)\right\|_0=1$, and the $U_{k}\left(s\right)$ are uniformly bounded with $U_0\equiv1$; hence there is a constant $M_0>0$ such that
\begin{equation*}
\left\|\doublehat{w}\left(\tilde{\tau}_n,\,\cdot\,\right) -\hat{w}\left(\tilde{\rho}_n ,\,\cdot\,\right)\right\|^2_0\leq\frac{M_0}{\tilde{\tau}_n^2}.
\end{equation*}
\item\label{item06sss}\textbf{Step 3.} Let ${w}\left({\rho}_n , \,\cdot\,\right)$ be an eigenfunction of the boundary-eigenvalue problem \eqref{eq_1sffss1} corresponding to an eigenvalue ${\lambda}_n=i{\rho}_n^2$. Then, in view of \eqref{eqasymff567}, it can be shown again via a direct calculation that 
\begin{equation*}
{\rho}_n=\tilde{\rho}_n
+\frac{ \frac{\beta^2\eta^2}{2}+2i\beta\eta
 }{4\tilde{\rho}_n}+\mathscr{O}\,(\tilde{\rho}_n^{-2}).
\end{equation*}
Furthermore,
\begin{equation*}
{w}\left({\rho}_n ,s\right)=\hat{w}\left(\tilde{\rho}_n ,s\right)\sum^{\infty}_{k=0}\frac{F_{k}\left(s\right)}{\tilde{\rho}_n^k}
\end{equation*}
where the $\hat{w}\left(\tilde{\rho}_n ,\,\cdot\,\right)$ are the eigenfunctions of \eqref{eq1ssd32} and the $F_{k}\left(s\right)$ are uniformly bounded with $F_0\equiv1$. There exists a constant $M_1>0$ such that 
\begin{equation*}
\left\|\hat{w}\left(\tilde{\rho}_n ,\,\cdot\,\right)-{w}\left({\rho}_n ,\,\cdot\,\right)\right\|^2_0\leq\frac{M_1}{\left|\tilde{\rho}_n\right|^2}.
\end{equation*}
Note that, for large enough $n$,
\begin{equation*}
\frac{\rho_n}{\tau_n}=\frac{\tilde{\rho}_n}{\tau_n}= \frac{\tilde{\tau}_n}{\tau_n}\sim1
\end{equation*}
and thus, again, we have that there exists a constant $M>0$ such that
\begin{equation*}
\left\|\doublehat{w}\left(\tilde{\tau}_n ,\,\cdot\,\right)-{w}\left({\rho}_n ,\,\cdot\,\right)\right\|^2_0\leq\frac{M}{\tilde{\tau}_n^2}.
\end{equation*}
Hence the sequence of eigenfunctions $\{w\left(\rho_n,\,\cdot\, \right){\}}_{n=1}^\infty$ is quadratically close to the orthonormal sequence $\{\doublehat{w}\left(\tilde{\tau}_n,\,\cdot\,\right){\}}_{n=1}^\infty$ in $\bm{L}_2\left(0,1\right)$. The desired result now follows from Bari's theorem.
\end{enumerate}
\end{proof}

\begin{remark}\label{R-4-1}
We point out that the eigenvalues and eigenfunctions of the boundary eigenvalue problem \eqref{eq_1sffss1} need not be in one-to-one correspondence. Indeed it is possible that a same eigenfunction corresponds to different eigenvalues of \eqref{eq_1sffss1}. For example, we know that \eqref{eq1ssd3cc2} is a selfadjoint problem in $\bm{L}_2$ with respect to the spectral parameter $\mu=-{\lambda}^2$ which admits only positive eigenvalues; namely, the differential operator
\begin{equation*}
Lw\coloneqq w^{(4)}-(\gamma-\eta^2)\,w''
\end{equation*}
is a positive-definite selfadjoint operator in $\bm{L}_2\left(0,1\right)$ defined on the $\lambda$-independent domain
\begin{equation*}
\bm{D}\left(L\right)=\left\{w\in\bm{L}_2\left(0,1\right)~\middle|
~\begin{gathered}
w\in{\bm{H}}^4\left(0,1\right), \\
w\left(0\right)=w''\left(0\right)=0,~ w''\left(1\right)=w^{(3)}\left(1\right)-(\gamma-\eta^2)\, w'\left(1\right)=0
\end{gathered}\right\}
\end{equation*}
which has eigenvalues $\mu_n$ and corresponding eigenfunctions ${w}\left(\mu_n , \,\cdot\,\right)$. To the eigenvalues $\lambda_n=i\sqrt{\mu_n}$, $\lambda_{-n}=\overline{\lambda_n}=-i\sqrt{\mu_n}$ of \eqref{eq1ssd3cc2} there corresponds the same eigenfunction in the sense that there is a function ${w}\left(\mu_n , \,\cdot\,\right)\in\bm{L}_2\left(0,1\right)$ such that $(\lambda_n I-L)\,{w}\left(\mu_n , \,\cdot\,\right)=0$ and $(\lambda_{-n} I-L)\,{w}\left(\mu_n , \,\cdot\,\right)=0$. This is because the spectral parameter transformation $\mu\mapsto -\lambda^2$ is not an injective mapping from the $\lambda$-plane to the $\mu$-plane. To determine the one-to-one nature of an eigenpair $\left\{\lambda,w\left(\lambda,\,\cdot\,\right)\right\}$, we must restrict $\lambda$ to lie in a subregion or sector of the plane. In the general case of differential operators associated with boundary-eigenvalue problems on a finite interval the partitioning method in \cite[Section II.4]{Naimark1967} can be used which consists of two steps. First, one adjusts the order of the spectral parameter in order to accommodate the order of the differential operator, say, with order $k$. This is analogous to the spectral parameter transformation $\lambda\mapsto i\rho^2$ in our case. Second, the $\rho$-plane is divided into $2k$ sectors
\begin{equation*}
\mathscr{S}_j\coloneqq\left\{\rho\in\mathbf{C}:\frac{j\pi}{k}\leq \arg\left(\rho\right) \leq\frac{\left(j+1\right)\pi}{k}\right\},\quad j=0,1,\ldots,2k-1,
\end{equation*}
such that injectivity with respect to $\rho^k$ obtains on $\mathscr{S}_{j}\cup \mathscr{S}_{j+1}$. On such a sector eigenpairs are one-to-one. For \eqref{eq_1sffss1}, the eigenpairs $\left\{\rho_n,w\left(\rho_n,\,\cdot\,\right)\right\}$ are one-to-one on $0<\arg\left(\rho\right)\le\frac{\pi}{2}$, or equivalently on $0<\arg\left(\lambda\right)\le\pi$. So the case we consider in the theorem, namely the Riesz basisness of the eigenfunctions $w\left(\rho_n,\,\cdot\,\right)$ in the space $\bm{L}_2\left(0,1\right)$, is shared by the eigenfunctions $\overline{w\left(\rho_n, \,\cdot\,\right)}$ corresponding to another sector.
\end{remark}

\begin{theorem}\label{T-4-4}
Let $\left\{\lambda_n,\overline{\lambda_n}\right\}_{n=1}^\infty$ be the zeros of $\Delta$ which satisfy \eqref{eqasymff567}, enumerated in such a way that $\lambda_{-n}=\overline{\lambda_n}$. Let $\{w\left(\rho_n, \,\cdot\,\right),\overline{w\left(\rho_n, \,\cdot\,\right)}{\}}_{n=1}^\infty$ be a sequence of eigenfunctions of \eqref{eq_1sffss1}, $\left\|w\left(\rho_n,\,\cdot\,\right)\right\|_0=1$, corresponding to $\{\lambda_{n},\lambda_{-n}{\}}_{n=1}^\infty$. Then the $\mathbf{CLS}$ operator $T$ has a sequence $\left\{\lambda_{\pm n}\right\}_{n=1}^\infty$ of eigenvalues to which there corresponds a sequence of eigenvectors $\left\{x_{\pm n}\right\}_{n=1}^\infty$ of $T$,
\begin{equation*}
x_{n}=\left(\begin{matrix}
\frac{w\left(\rho_n, \,\cdot\,\right)}{\lambda_n}\\
w\left(\rho_n, \,\cdot\,\right)
\end{matrix}\right),\quad x_{-n}=\left(\begin{matrix}
\frac{\overline{w\left(\rho_n, \,\cdot\,\right)}}{\overline{\lambda_n}}\\
\overline{w\left(\rho_n, \,\cdot\,\right)}
\end{matrix}\right),
\end{equation*}
having the following properties:
\begin{enumerate}[\normalfont(1)]
\item\label{T-3-2-bss} $(\lambda_n I-T)\,x_n=0$, $(\lambda_{-n} I-T)\,x_{-n}=0$, and
\begin{equation}\label{dddfeq123}
\left\|x_{\pm n}\right\|\asymp\sqrt{2},\quad n\in\mathbf{N}.
\end{equation}
\item\label{T-3-2-css} $\left\{x_{\pm n}\right\}_{n=1}^\infty$ forms a Riesz basis for $\mathbb{X}$.
\end{enumerate}
\end{theorem}
\begin{proof}
For the proof of property \ref{T-3-2-bss} note that if $T$ has an eigenvalue $\lambda_n$ with corresponding eigenvector $x_n$, we have
\begin{equation*}
Tx_n=\left(\begin{matrix}
w\left(\rho_n, \,\cdot\,\right)\\
-\frac{w^{(4)}\left(\rho_n, \,\cdot\,\right)}{\lambda_n}+(\gamma-\eta^2)\, \frac{w''\left(\rho_n, \,\cdot\,\right)}{\lambda_n}-2\beta\eta w'\left(\rho_n, \,\cdot\,\right)
\end{matrix}\right)=\left(\begin{matrix}
w\left(\rho_n, \,\cdot\,\right)\\
\lambda_n w\left(\rho_n, \,\cdot\,\right)
\end{matrix}\right)=\lambda_nx_{n}
\end{equation*}
and, since $\left(\lambda_n I-T\right)x_n=0$,
\begin{equation*}
(\lambda_{-n} I-T)\,x_{-n}=(\overline{\lambda_n} I-T)\,\overline{x_{n}}=\overline{({\lambda_n} I-T)\,x_{n}}=0.
\end{equation*}
To verify the normalisation condition \eqref{dddfeq123} and the remaining property \ref{T-3-2-css} of the theorem, we first compute in the sequence $\left\{x_{\pm n}\right\}_{n=1}^\infty$
\begingroup
\allowdisplaybreaks
\begin{align*}
\left\|x_n\right\|^2&=\frac{1}{\left|\lambda_n\right|^2}\int^1_0 \left|w''\left(\rho_n,s\right)\right|^2ds+\frac{\gamma-\eta^2}{\left|\lambda_n\right|^2}\int^1_0\left|w'\left(\rho_n,s\right)\right|^2ds+\int^1_0\left|w\left(\rho_n,s\right)\right|^2ds\\
&=-\frac{\overline{\lambda_n}}{\left|\lambda_n\right|^2}\kappa\left|w'\left(\rho_n,1\right)\right|^2-\frac{2\beta\eta\overline{\lambda_n}}{\left|\lambda_n\right|^2}\int^1_0w\left(\rho_n,s\right)\overline{w'\left(\rho_n,s\right)}\,ds\\
&\quad+\frac{\left|\lambda_n\right|^2-\overline{\lambda_n^2}}{\left|\lambda_n\right|^2}\int^1_0\left|w\left(\rho_n,s\right)\right|^2ds,
\end{align*}
\endgroup
the second line following immediately when we integrate by parts and then use the boundary conditions and differential equation in \eqref{eq_1sffss1}. Also,
\begin{equation*}
\left\|x_n\right\|^2=-\frac{\lambda_n}{\left|\lambda_n\right|^2}\kappa\left|w'\left(\rho_n,1\right)\right|^2-\frac{2\beta\eta\lambda_n}{\left|\lambda_n\right|^2}\int^1_0w'\left(\rho_n,s\right)\overline{w\left(\rho_n,s\right)}\,ds+\frac{\left|\lambda_n\right|^2-\lambda_n^2}{\left|\lambda_n\right|^2}\int^1_0\left|w\left(\rho_n,s\right)\right|^2ds.
\end{equation*}
Thus we have
\begin{equation}
\begin{split}
\left\|x_n\right\|^2&=-\frac{\operatorname{Re}\,(\overline{{\lambda}_n})}{\left|\lambda_n\right|^2}\left(\beta\eta \left|w\left(\rho_n,1\right)\right|^2+\kappa \left|w'\left(\rho_n,1\right)\right|^2\right)\\
&\quad+\frac{2\beta\eta\operatorname{Im}\,(\overline{\lambda_n})}{\left|\lambda_n\right|^2}\int^1_0\operatorname{Im}\,(w\left(\rho_n,s\right)\overline{w'\left(\rho_n,s\right)})\,ds+\frac{2\left(\operatorname{Im}\,(\lambda_n)\right)^2}{\left|\lambda_n\right|^2}\int^1_0 \left|w\left(\rho_n,s\right)\right|^2ds.\label{eqwwr45}
\end{split}
\end{equation}
Then after some standard estimates of $w\left(\rho_n,\,\cdot\,\right)$, $w'\left(\rho_n,\,\cdot\,\right)$ applied to the boundary terms and the first integral in \eqref{eqwwr45} we obtain that when $\left\|{w}\left(\rho_n, \,\cdot\,\right)\right\|_0=1$, there exists a constant $M_0>0$ such that
\begin{equation*}
\frac{\left|\operatorname{Re}\,(\overline{{\lambda}_n})\right|}{\left|\lambda_n\right|^2}\left(\beta\eta \left|w\left(\rho_n,1\right)\right|^2+\kappa \left|w'\left(\rho_n,1\right)\right|^2\right)+\frac{2\beta\eta\left|\operatorname{Im}\,(\overline{\lambda_n})\right|}{\left|\lambda_n\right|^2}\int^1_0\left|w\left(\rho_n,s\right)\overline{w'\left(\rho_n,s\right)}\right|ds\le\frac{M_0}{\left|\rho_n\right|^2}
\end{equation*}
and \eqref{eqwwr45} yields $\frac{\left(\operatorname{Im}\,(\lambda_n)\right)^2}{\left|\lambda_n\right|^2}\asymp 1$, and hence
\begin{equation*}
\left\|x_n\right\|^2\asymp2\int^1_0 \left|w\left(\rho_n,s\right)\right|^2ds=2.
\end{equation*}
Set now
\begin{equation*}
\doublehat{x}_{n}=\left(\begin{matrix}
\frac{\doublehat{w}\left(\tilde{\tau}_n, \,\cdot\,\right)}{i\tilde{\tau}^2_n}\\
\doublehat{w}\left(\tilde{\tau}_n, \,\cdot\,\right)
\end{matrix}\right),\quad\doublehat{x}_{-n}=\left(\begin{matrix}
-\frac{{\doublehat{w}\left(\tilde{\tau}_n, \,\cdot\,\right)}}{i{\tilde{\tau}^2_n}}\\
{\doublehat{w}\left(\tilde{\tau}_n, \,\cdot\,\right)}
\end{matrix}\right),\quad n\in\mathbf{N},
\end{equation*}
where we recall from the proof of Theorem \ref{T-4-2} that $\{\tilde{\tau}_n,\doublehat{w}\left(\tilde{\tau}_n, \,\cdot\,\right)\}$ are the (real) eigenpairs of the selfadjoint boundary-eigenvalue problem \eqref{eq1ssd3cc2}. We compute, for any $\doublehat{x}_{n},\doublehat{x}_{m}\in\mathbb{X}$,
\begingroup
\allowdisplaybreaks
\begin{align*}
\left<\doublehat{x}_{n},\doublehat{x}_{m}\right>&=\frac{1}{\tilde{\tau}_n^2\tilde{\tau}_m^2}\int^1_0 \doublehat{w}''\left(\tilde{\tau}_n, s\right){\doublehat{w}''\left(\tilde{\tau}_m, s\right)}\,ds+\frac{\gamma-\eta^2}{\tilde{\tau}_n^2\tilde{\tau}_m^2}\int^1_0\doublehat{w}'\left(\tilde{\tau}_n, s\right){\doublehat{w}'\left(\tilde{\tau}_m, s\right)}\,ds\\
&\qquad+\int^1_0\doublehat{w}\left(\tilde{\tau}_n, s\right){\doublehat{w}\left(\tilde{\tau}_m, s\right)}\,ds\\
&=\frac{\tilde{\tau}_m^4}{\tilde{\tau}_n^2\tilde{\tau}_m^2}\int_0^1\doublehat{w}\left(\tilde{\tau}_n, s\right){\doublehat{w}\left(\tilde{\tau}_m, s\right)}\,ds+\int_0^1\doublehat{w}\left(\tilde{\tau}_n, s\right){\doublehat{w}\left(\tilde{\tau}_m, s\right)}\,ds
\end{align*}
\endgroup
after the usual integrations by parts and using the boundary conditions and differential equation in \eqref{eq1ssd3cc2}. Thus for $m=n$, when $\|\doublehat{w}\left(\tilde{\tau}_n, \,\cdot\,\right)\!{\|}_0=1$, we get $\left\|\doublehat{x}_n\right\|^2=2$. So there exists a sequence $\{\doublehat{x}_{\pm n}{\}}_{n=1}^\infty$ which forms an orthonormal basis for $\mathbb{X}$, and there exists a constant $M>0$ such that
\begingroup
\allowdisplaybreaks
\begin{align*}
\left\|\doublehat{x}_n-x_n\right\|^2&=\int_0^1\left|\frac{\doublehat{w}''\left(\tilde{\tau}_n,s\right)}{i\tilde{\tau}_n^2}-\frac{w''\left(\rho_n,s\right)}{\lambda_n}\right|^2ds+(\gamma-\eta^2)\int_0^1\left|\frac{\doublehat{w}'\left(\tilde{\tau}_n,s\right)}{i\tilde{\tau}_n^2}-\frac{w'\left(\rho_n,s\right)}{\lambda_n}\right|^2ds\\
&\qquad+\int_0^1\left|\doublehat{w}\left(\tilde{\tau}_n,s\right)-w\left(\rho_n,s\right)\right|^2ds\\
&\leq\frac{M}{\tilde{\tau}_n^2}\int_0^1\left|\doublehat{w}\left(\tilde{\tau}_n, s\right)\right|^2ds=\frac{M}{\tilde{\tau}_n^2}.
\end{align*}
\endgroup
We then obtain that the two sequences $\{\doublehat{x}_{\pm n}{\}}_{n=1}^\infty$ and $\{{x}_{\pm n}{\}}_{n=1}^\infty$ are quadratically close. With this, again applying Bari's theorem, the proof of property \ref{T-3-2-css}, and thus of the theorem is complete.
\end{proof}

\begin{remark}
It should be noted that the Riesz basisness result obtained in Theorem \ref{T-4-4} can be strengthened to yield Bari basisness of the eigenvectors for $\mathbb{X}$. This is because the boundary-eigenvalue problem \eqref{eq_1sffss1} under the spectral transformation $\lambda\mapsto i\rho^2$ is Birkhoff regular, and by Theorem \ref{L-2-1}, $T^{-1}$ is compact, and so $T$ has a purely discrete spectrum, which guarantees the minimality of the eigenvectors in $\mathbb{X}$. In fact the completeness of the eigenvectors can be verified by application of Keldysh's results \cite[Sections V.8--10]{GohbergKrein1969}. (See \cite[Section 6]{MR3588930} for explanation.)
\end{remark}

\section{Series expansions and exponential stability}\label{sec_4b}

We begin by characterising the series expansion of the solution \eqref{eq_03} to the $\mathbf{CLS}$ represented by the initial-value problem \eqref{eq_03xab}. On the basis of Theorems \ref{L-2-1} and \ref{T-4-4} and Corollary \ref{rem01}, the following theorem is valid.
\begin{theorem}\label{T-4-s}
Given $x_0\in \bm{D}\left(A\right)$, the solution \eqref{eq_03} to \eqref{eq_03xab} can be represented in series form as
\begin{equation*}
\mathcal{U}\left(t\right)x_0=\sum^{\infty}_{n=1}e^{\lambda_n t}\,\bigl<x_0, z_n\bigr>\,x_n+\sum^{\infty}_{n=1}e^{\lambda_{-n} t}\,\bigl<x_0, z_{-n}\bigr>\,x_{-n},\quad t\ge 0,
\end{equation*}
where $\bigl\{x_{\pm n},z_{\pm n}\bigr\}_{n=1}^\infty$ is a biorthogonal sequence in $\mathbb{X}$ and $\mathcal{U}\left(t\right)$ is the contraction $C_0$-semigroup generated by the $\mathbf{CLS}$ operator $T$, which extends to a $C_0$-group on $\mathbb{X}$.
\end{theorem}
\begin{remark}\label{rmrkfin01}
By the results of Theorem \ref{T-4-1} and Corollary \ref{rem01} it is clear that there can be  at most a finite number of nonsimple eigenvalues of $T$. It is not easy to analyse the multiplicities such eigenvalues directly, despite the positive results on the simplicity of nonreal eigenvalues. For completeness we note that Theorem \ref{T-4-s} could be extended to cover the case in which generally there exists a finite number of nonsimple eigenvalues of $T$. For some integer $m$ let $\left\{\mu_j,\overline{\mu_j}\right\}_{j=1}^m$ be the zeros of $\Delta$ corresponding to these eigenvalues for which $\nu\left(\mu_j\right)>1$. Let $\left\{E\left(\mu_j,T\right),E\left(\overline{\mu_j},T\right)\right\}_{j=1}^m$ be the corresponding eigenprojections (Riesz projections) of $T$. Then, for $t\ge 0$,
\begingroup
\allowdisplaybreaks
\begin{align*}
&\mathcal{U}\left(t\right)x_0\\
&\quad=\sum_{j=1}^me^{\mu_j t}\sum_{k=0}^{\nu\left(\mu_j\right)}\frac{t^k}{k!}\left(T-\mu_j I\right)^kE\left(\mu_j,T\right)x_0+\sum_{j=1}^me^{{\mu_{-j}} t}\sum_{k=0}^{\nu\left({\mu_{-j}}\right)}\frac{t^k}{k!}\left(T-{\mu_{-j}} I\right)^kE\left({\mu_{-j}},T\right)x_0\\
&\qquad+\sum^{\infty}_{n=1}e^{\lambda_n t}\,\bigl<x_0, z_n\bigr>\,x_n+\sum^{\infty}_{n=1}e^{\lambda_{-n} t}\,\bigl<x_0, z_{-n}\bigr>\,x_{-n},
\end{align*}
\endgroup
the series expansion being true in $\mathbb{X}$ in the sense of norm-convergence.
\end{remark}
Combining Theorem \ref{T-4-s} with the results of Section \ref{sec_3}, we obtain as the final result exponential stability of the $\mathbf{CLS}$.
\begin{theorem}
The $\mathbf{CLS}$ is exponentially stable for $\eta,\kappa>0$ and $\gamma> \eta^2$.
\end{theorem}

\section{Conclusions}\label{sec_5}

This paper has proven the exponential stability of the $\mathbf{CLS}$ consisting of the one-dimensional model of a stretched tube conveying fluid, with one end simply supported, via boundary control at the other end. Specifically we have shown that when the tension in the tube is greater than the square of the fluid-flow velocity, in the sense that $\gamma> \eta^2$, the addition, through the feedback relation \eqref{eq_06ssyy}, of small boundary damping represented by $\kappa>0$ does not destroy (and actually improves) the exponential stability of the $\mathbf{CLS}$, even when flow is admitted -- i.e., when $\eta>0$. For the purpose of identifying larger parameter regions in which all vibrations die out exponentially fast and the tube neither flutters nor buckles, this is a physically interesting conclusion.

In the proof process, the so-called spectral approach has been followed involving a thorough analysis of the spectral properties and of the Riesz basisness for the eigenvectors of the spectral problems and operators involved. We have established, somewhat independently, the Riesz basisness of the eigenfunctions of the boundary-eigenvalue problem \eqref{eq_1sffss1} in the space $\bm{L}_2\left(0,1\right)$. We have also shown that the eigenvectors of the $\mathbf{CLS}$ operator $T$ constitute a Riesz basis for the state space $\mathbb{X}$. Under these circumstances the spectrum of the $\mathbf{CLS}$ operator determines the energy decay rate $\varepsilon$ in \eqref{eq_05}, as has all been explained in the paper; hence, since we have proven $\varepsilon>0$, the exponential decay of the energy of solutions of the $\mathbf{CLS}$, i.e.\ exponential stability of the $\mathbf{CLS}$.

The boundary-eigenvalue problem \eqref{eq_1sffss1} considered in this paper belongs to the general class of boundary-eigenvalue problems with $\lambda$-dependent boundary conditions. Associated with \eqref{eq_1sffss1} and its linearisation \eqref{eq_1sffss1ss} have been the operator pencils
\begin{equation*}
\mathcal{M}\left(\lambda\right)=\lambda^2G+\lambda D+C\quad\text{and}\quad\mathcal{P}\left(\lambda\right)=\lambda I-T
\end{equation*}
in the Hilbert product spaces $\mathbb{Y}$ and $\mathbb{X}$, respectively. So there arises the more general question as to the appropriate choice of underlying spaces $\mathbb{Y}$ and $\mathbb{X}$ and on conditions on the operators involved, under which the Riesz basisness of the root vectors of one pencil does follow automatically from that of the other. In this paper we have not, for example, addressed the question whether the eigenvectors of $\mathcal{M}$ are a Riesz basis for $\mathbb{Y}$. This problem and related ones will be addressed in a future paper.

\bigskip\noindent
\textbf{Acknowledgments.} This research was supported in part by the National Natural Science Foundation of China under Grant NSFC-61773277 and in part by the Stiftung KESSLER+CO für Bildung und Kultur und Grant EXPLOR-24MM. The authors would like to thank Dr.\ David Swailes of the School of Mathematics, Statistics and Physics, Newcastle University for useful discussions and remarks. They would also like to thank one of the referees for pointing out that the result of Theorem \ref{T-4-1} for the case where $\eta=0$ is essentially contained in \cite[Theorem 5.2]{MR3725254}.

\bibliographystyle{plain}
\bibliography{BibLio02}

\end{document}